\documentclass[reqno, letter, 11pt]{amsart}
\usepackage[utf8]{inputenc}

\usepackage{amssymb}
\usepackage{graphicx}
\usepackage{latexsym}
\usepackage{stmaryrd}
\usepackage{enumitem}
\usepackage[bookmarks]{hyperref}
\usepackage{multirow}
\usepackage{xcolor}
\usepackage{relsize}
\usepackage{comment}
\usepackage{xifthen}
\usepackage{mathrsfs}
\usepackage{svg}
\usepackage[fleqn]{nccmath}
\usepackage{mathtools}
\usepackage{tikz-cd}
\colorlet{darkishRed}{red!80!black}
\colorlet{darkishBlue}{blue!60!black}
\colorlet{darkishGreen}{green!60!black}
\renewcommand{\leq}{\leqslant}
\renewcommand{\geq}{\geqslant}
\renewcommand{\ge}{\geq}
\renewcommand{\le}{\leq}
\usepackage{accents}
\newcommand{\underring}[1]{\underaccent{\hbox to 0pt{\hss\normalfont\kern.1em \r{}\hss}}{#1}}

\usepackage{hyperref}
\hypersetup{
    unicode,
    colorlinks,
    linkcolor={red!60!black},
    citecolor={green!60!black},
    urlcolor={blue!60!black}
}
\usepackage[nameinlink, capitalise, noabbrev]{cleveref}

\DeclareMathOperator{\medcup}{\mathsmaller{\bigcup}}

\renewcommand{\subset}{\subseteq}
\renewcommand{\supset}{\supseteq}
\newcommand{\rest}{\restriction}

\newcommand{\Abs}[1]{\partial_{\Omega} {#1}}

\newcommand{ \N } { \mathbb{N} }
\newcommand{ \Z } { \mathbb{Z} }

\newcommand{\eps}{\varepsilon}
\def\downcl#1{\lceil{#1}\rceil}
\def\upcl#1{\lfloor{#1}\rfloor}

\makeatletter

\def\calCommandfactory#1{%
   \expandafter\def\csname c#1\endcsname{\mathcal{#1}}}
\def\frakCommandfactory#1{%
   \expandafter\def\csname frak#1\endcsname{\mathfrak{#1}}}

\newcounter{ctr}
\loop
  \stepcounter{ctr}
  \edef\X{\@Alph\c@ctr}
  \expandafter\calCommandfactory\X
  \expandafter\frakCommandfactory\X
  \edef\Y{\@alph\c@ctr}
  \expandafter\frakCommandfactory\Y
\ifnum\thectr<26
\repeat

\renewcommand{\cC}{\mathscr{C}}

\newcommand{\tame}{uniform}
\newcommand{\ff}{sequentially faithful}
\newcommand{\Ff}{Sequentially faithful}
\newcommand{\pt}{partition tree}

\newcommand{\HR}{\cR}

\usepackage{soul}

\newcommand{\closure}[1]{\overline{#1}}

\newcommand{\Set}[1]{{\left\lbrace {#1} \right\rbrace}}

\def\set#1:#2{\Set{{#1} \colon {#2}}}

\newcommand{\dc}[1]{\lceil #1\rceil}
\newcommand{\uc}[1]{\lfloor #1\rfloor}

\newtheorem{theorem}{Theorem}[section]

\newtheorem{mainresult}{Theorem}
\newtheorem{claim}{Claim}

\newtheorem{proposition}[theorem]{Proposition}
\newtheorem{corollary}[theorem]{Corollary}
\newtheorem{lemma}[theorem]{Lemma}

\newtheorem{prob}[theorem]{Open problem}

\newtheorem{innercustomthm}{Theorem}
\newenvironment{customthm}[1]
  {\renewcommand\theinnercustomthm{#1}\innercustomthm}
  {\endinnercustomthm}

\theoremstyle{definition}

\newtheorem{definition}[theorem]{Definition}

\theoremstyle{remark}

\setenumerate{label={\normalfont (\roman*)},itemsep=0pt}

\usepackage{geometry}
\begin{document}
\title{\mbox{A representation theorem for end spaces of infinite graphs}}

\author{Jan Kurkofka}
\address{Universität Hamburg, Department of Mathematics, Bundesstraße 55 (Geomatikum), 20146 Hamburg, Germany}
\curraddr{University of Birmingham, School of Mathematics, Birmingham B15 2TT, UK}
\email{jan.kurkofka@uni-hamburg.de, j.kurkofka@bham.ac.uk}

\author{Max Pitz}
\address{Universität Hamburg, Department of Mathematics, Bundesstraße 55 (Geomatikum), 20146 Hamburg, Germany}
\email{max.pitz@uni-hamburg.de}

\begin{abstract}
End-spaces of infinite graphs naturally generalise the Freudenthal boundary and sit at the interface between graph theory, geometric group theory and topology. 

Our main result is that every end-space can topologically be represented by a special order tree. 
Our main proof ingredient is a structure theorem that we introduce, which carves out the order-tree-like structure of any graph in such a way that there is a natural bijection between the ends of the graph and the limit-type down-closed chains of the order-tree.

\end{abstract}



\maketitle


\section{Introduction}

 Intuitively, the ends of a graph capture the different directions in which the graph expands towards infinity. The simplest case is given by the class of (rooted) trees. Here, the ends correspond to the different rays starting at the root, and their end spaces are precisely the completely ultrametrizable spaces. 
 The concept of an end generalises naturally from trees to arbitrary graphs. Following Halin \cite{Halin_Enden64}, an end of a graph $G$ is an equivalence class of rays where two rays are equivalent if for every finite set $X$ of vertices they eventually belong to the same component of $G-X$. 
The set of ends of $G$ can be turned into the \emph{end-space} $\Omega(G)$ by equipping it with the following topology:
for every finite $X\subset V(G)$ and every component $C$ of $G-X$, one declares the set of ends represented by the rays in~$C$ to be basic open, and takes the topology generated by all these basic open sets.
For example, infinite cliques and the $\Z\times\Z$ grid have a single end, while the end-space of the infinite binary tree is homeomorphic to the Cantor set.

Ends play a crucial role in extending results about finite graphs to infinite graphs, see for example~\cite{aurichi2024edgeconnectivityedgeendsinfinitegraphs,CanonicalInfiniteToT,diestel2011locally,InfiniteFleischner,InfiniteKconChar,CountableLovaszCherkassky}.
Perhaps surprisingly, ends have recently been used for the converse as well: by studying the structure of finite graphs through their covering spaces~\cite{Local2Separators,carmesin2024stallingstypetheoremfinitegroups,GraphDec,CayleyCovers}.
In~Bass-Serre theory, Stallings' theorem uses ends to detect product-structure in groups such as amalgamated free products or HNN-extensions.


However, despite significant advances in our  understanding of end spaces \cite{diestel2006end,diestel2003graph,FirstSecondCountable,ApproximatingNormalTrees, polat1996ends,polat1996ends2,sprussel2008end}, we still have no complete picture of their precise topological characteristics.
Indeed, end spaces of graphs are significantly more complex than end spaces of trees.  
For example, let $G$ be obtained from an uncountable clique $K$ by disjointly adding new rays~$R_v$, one for each vertex~$v$ in $K$, and identifying the first vertex of each $R_v$ with~$v$.
Then $\Omega(G)$ is not first countable at the end that contains the rays in~$K$, and hence $\Omega(G)$ is not metrizable. Thus, end spaces of graphs cannot be represented by trees, and the following natural question arises: is there a narrow class of graphs just wide enough to represent as their end spaces all topological spaces that arise as end spaces of graphs? 

Our representation theorem offers an affirmative answer to this question. We recall the following concepts:
An \emph{order-tree} is a partial order $T=(T,\le)$ with a least element called the root such that all chains in $T$ are well-ordered by~$\le$. An order-tree $T$ is \emph{special} if $T$ can be partitioned into countably many antichains. A \emph{path} in $T$ is a down-closed chain. A path without maximal element is a \emph{high-ray}. By identifying paths with their characteristic function, the set $\cP(T)$ of all paths of $T$ is a closed, hence compact subspace of $\Set{0,1}^{T}$. The space $\cP(T)$ is the \emph{path space} associated with $T$. The \emph{ray space} $\cR(T)$ is the subspace of $\cP(T)$ consisting of all high-rays of $T$. 

Brochet and Diestel~\cite{brochet1994normal} formalised when a graph resembles a given order-tree:
A~\emph{$T$-graph} is a graph with vertex-set~$T$ in which edges run between comparable vertices only, so that each successor node is adjacent to its immediate predecessor, and so that limit nodes are adjacent to cofinally many of their predecessors.
See \cite{pitz_survey} for a survey.
Diestel and Leader defined that a $T$-graph $G$ is \emph{uniform} if for every limit-node $t\in T$ there are finitely many nodes $t_1<\cdots <t_n$ below $t$ such that every edge between a node $<t$ and a node $>t$ starts in a~$t_i$; see~\cite{DiestelLeaderNST}.

\begin{mainresult}[Representation theorem]
\label{thm_rep}
\label{thm:EndspaceTgraph}
The following are equivalent for a topological space $X$:
	\begin{enumerate}[label={\textnormal{(\arabic*)}}]
		\item $X$ is homeomorphic to the end space of a graph,
  \item $X$ is homeomorphic to the end space of a uniform $T$-graph on a (special) order tree $T$,
		\item $X$ is homeomorphic to the ray space of a special order tree $T$.
 	\end{enumerate}
\end{mainresult}

As $T$-graphs on special order trees do not contain uncountable clique minors, we conclude as an application of (2) that every end space can be represented by a graph without uncountable clique minor. See \cref{sec_8} for further applications.

Assertion (3) has a different flavour, as it describes a topological space associated with an order tree to represent end spaces of graphs. By moving from graphs to order trees, we trade ends, i.e.\ fairly complicated equivalence classes of rays, for high-rays in $T$, obtaining a canonical representation system for the ends of the graph if you will. Diestel asked in 1992 for a topological characterisation of end-spaces of graphs~\cite[5.1]{diestel1992end}. The statement of (3) will be used by the second author as a key ingredient to give such a characterisation~\cite{pitz2023}.

\medskip

We now describe some of the ingredients that go into the proof of \cref{thm_rep}. The implication $(3) \Rightarrow (2)$ follows from a canonical method due to Diestel and Leader \cite[Theorem~6.2]{DiestelLeaderNST} to build, from a special order tree $T$ with antichain partition $\set{U_n}:{n \in \N}$, a \emph{\tame} graph $G$ on $T$: Successors in $T$ are connected to their predecessors, and given a limit $t \in T$, recursively pick down-neighbours $t_0 < t_1 < t_2 < \cdots<  t$ with $t_i \in U_{n_i}$  so that  each $n_{i}$ is smallest  possible subject to $t_{i-1} < t_i < t$. See \cref{prop_specialtrees} below for a thorough discussion. While $G$ depends on the exact choice of $\set{U_n}:{n \in \N}$, we show that any two \tame\ graphs on the same order tree $T$ have end spaces naturally homeomorphic to each other and to $\cR(T)$. 
For the converse $(2) \Rightarrow (3)$, we additionally prove that the existence of a \tame\ graph on an order tree $T$ implies that $T$ must be special. 

The implication $(2) \Rightarrow (1)$ is trivial, and the main difficulty in the proof of \cref{thm:EndspaceTgraph} is the implication $(1) \Rightarrow (2)$, i.e.\ to construct for an arbitrary graph $G$ a uniform $T$-graph $H$ on a special order-tree~$T$ such that $\Omega(G)$ is homeomorphic to~$\Omega(H)$. Indeed, the majority of this paper is concerned with the proof of this implication.
For this, we prove a structure theorem for all connected graphs in terms of \emph{\pt s}, a new concept that we introduce in this paper and which is inspired by the work of Brochet and Diestel~\cite{brochet1994normal}.
Roughly speaking, a \pt\ is a substructure of a graph $G$ that takes the shape of an order-tree $T$ and combines key aspects of both algorithmic depth-first search trees and tree-decompositions from the theory of graph-minors.
Formally, a \emph{\pt } of $G$ is a pair $(T,\cV)$ of an order-tree $T=(T,\le)$ and a partition $\cV=(V_t : t\in T)$ of $V(G)$ into connected vertex-sets~$V_t$ such that
\begin{enumerate}
    \item the contraction-minor\footnote{Recall that $H$ is a \emph{contraction-minor} of $G$ if $V(G)$ can be partitioned into connected vertex sets $V_h$ for $h\in V(H)$ such that $hh'$ is an edge in $H$ if and only if $G$ contains an edge between $V_h$ and~$V_{h'}$.} $G/\cV$ is a $T$-graph;
    \item for every successor-node $t\in T$ the neighbourhood of $\bigcup_{t'\ge t}V_{t'}$ in $G$ is finite; and
    \item $|V_t|=1$ for all non-limits~$t\in T$.
\end{enumerate}
\noindent 
Roughly speaking, we say that $(T,\cV)$ \emph{displays} the ends of $G$ if a canonical correspondence between the ends of $G$ and the high-rays of~$T$ is bijective.
And we say that $(T,\cV)$ is \emph{\ff } if this canonical correspondence respects convergence of sequences.
The structure theorem states:

\begin{mainresult}\label{thm:PartitionTree}
Every connected $G$ has a \ff\ \pt\ that displays all its ends.
\end{mainresult}

Besides being a crucial step in proving \cref{thm_rep}, this structure theorem may prove useful in other situations, too. In \cref{sec_8} we present a variation on Halin's \emph{end-faithful spanning tree conjecture}, as well as a short derivation of Carmesin's result \cite[5.17]{carmesin2014all} that every graph admits a nested set of finite-order separations that distinguishes every two ends of the graph.

As an ingredient to the proof of \cref{thm:PartitionTree}, we introduce a new method called \emph{enveloping} an arbitrary vertex-set $U\subset V(G)$, which intuitively enables one to maximally expand $U$ without changing which ends can be separated from $U$ by finitely many vertices. 
This \emph{envelope method} has already found further applications since the completion of our paper, see e.g.\ \cite{ aurichi2024topologicalremarksendedgeend, koloschin2022end, pitz_shortCarmesin}.

\subsection*{Organisation of this paper} 
\cref{sec_2} recalls notation and facts on end spaces and $T$-graphs.
In \cref{sec_envelopes} we introduce the method of enveloping vertex-sets.
In \cref{sec_3} we discuss an alternative perspective on uniform $T$-graphs and establish an equivalence between uniformity and special order-trees. 
In \cref{sec_4} we give investigate end-spaces of uniform $T$-graphs in terms of~$T$, and show that they are naturally homeomorphic to the ray space $\cR(T)$, completing the proof of the equivalence $(2) \Leftrightarrow (3)$ in \cref{thm_rep}.
In \cref{sec_5} we introduce partition trees and prove core properties.
We prove the structure theorem, \cref{thm:PartitionTree}, in \cref{sec_6}.
In \cref{sec_7}, we derive from \cref{thm:PartitionTree} the remaining implication  $(1) \Rightarrow (2)$ in \cref{thm_rep}. 
Lastly, in \cref{sec_8} we have included some applications.

\newpage
\section{End spaces and \texorpdfstring{$T$}{T}-graphs: a reminder}
\label{sec_2}

\noindent For graph theoretic terms we follow the terminology in \cite{Bible}, and in particular \cite[Chapter~8]{Bible} for ends of graphs and the end spaces~$\Omega(G)$. A function $f \colon X \to Y$ is \emph{finite-to-one} if $f^{-1}(y)$ is finite for all $y \in Y$.

\subsection{End spaces}
\label{sec_endsdefns}
A $1$-way infinite path is called a \emph{ray} and the subrays of a ray are its \emph{tails}. Two rays in a graph $G = (V,E)$ are \emph{equivalent} if no finite set of vertices separates them; the corresponding equivalence classes of rays are the \emph{ends} of $G$. 
The set of ends of a graph $G$ is denoted by $\Omega = \Omega(G)$. 
Usually, ends of graphs are denoted by~$\omega$, but here we will denote ends by~$\eps$ or~$\eta$ to avoid confusion with ordinals such as~$\omega_1$.
If $X \subseteq V$ is finite and $\eps \in \Omega$ is an end, 
there is a unique component of $G-X$ that contains a tail of every ray in $\eps$, which we denote by $C(X,\eps)$. 
Then $\eps$ \emph{lives} in the component $C(X,\eps)$.
If $C$ is any component of $G-X$, we write $\Omega(X,C)$ for the set of ends $\eps$ of $G$ with $C(X,\eps) = C$, and abbreviate $\Omega(X,\eps) := \Omega(X,C(X,\eps))$. 
Finally, if $\cC$ is any collection of components of $G-X$, we write $\Omega(X,\cC) := \bigcup\,\{\,\Omega(X,C) \colon C \in \cC\,\}$.

The collection of all sets $\Omega(X,C)$ with $X\subset V$ finite and $C$ a component of $G-X$ forms a basis for a topology on $\Omega$.
This topology is Hausdorff, and it is \emph{zero-dimensional} in that it has a basis consisting of closed-and-open sets. 

A crucial property of end spaces is that they are \emph{Fréchet-Urysohn}: This means that closures are given by convergent sequence, i.e.
 $x \in \closure{X}$ for some subset $X \subset \Omega(G)$ if and only if there are $x_n \in X$ with $x_n \to x$ (as $n \to \infty$). In particular, this implies that a function $f$ between two end spaces is continuous if and only if $f$ is sequentially continuous; later in this paper, we will therefore only check for sequential continuity.
 
 Note that when considering end spaces $\Omega(G)$, we may always assume that $G$ is connected; adding one new vertex and choosing a neighbour for it in each component does not affect the end space.

Recall that a \emph{comb} is the union of a ray $R$ (the comb's \emph{spine}) with infinitely many disjoint finite paths, possibly trivial, that have precisely their first vertex on~$R$. 
The last vertices of those paths are the \emph{teeth} of this comb.
Given a vertex set $U$, a \emph{comb attached to} $U$ is a comb with all its teeth in $U$, and a \emph{star attached to} $U$ is a subdivided infinite star with all its leaves in $U$.
Then the set of teeth is the \emph{attachment set} of the comb, and the set of leaves is the \emph{attachment set} of the star.
\begin{lemma}[Star-comb lemma]\label{StarCombLemma}
Let $U$ be an infinite set of vertices in a connected graph $G$.
Then $G$ contains either a comb 
attached to~$U$ or a star attached to~$U$.
\end{lemma}

Let us say that an end $\eps$ of $G$ is contained \emph{in the closure} of $M$, where $M$ is either  a subgraph of $G$ or a set of vertices of $G$, if for every finite vertex set $X\subset V$ the component $C(X,\eps)$ meets $M$.
Equivalently, $\eps$~lies in the closure of $M$ if and only if $G$ contains a comb attached to $M$ with its spine in~$\eps$.
We write $\Abs{M}$ 
for the subset of $\Omega$ that consists of the ends of $G$ lying in the closure of $M$.

\subsection{Normal trees}

Given a graph $H$, we call a path $P$ an \emph{$H$-path} if $P$ is non-trivial and meets $H$ exactly in its endvertices. 
In particular, the edge of any $H$-path of length~$1$
is never an edge of~$H$.
A rooted tree $T\subset G$ is \emph{normal} in $G$ if the endvertices of every $T$-path in~$G$ are comparable in the tree-order of $T$, cf.~\cite{Bible}.
The rays of a normal tree~$T$ that start at its root are its \emph{normal rays}.

Since their invention due to Jung~\cite{jung1968wurzelbaume},
normal trees have developed to be perhaps the single most useful structural tool in infinite graph theory. Positive results include a characterisation due to Jung when normal spanning trees exist \cite{jung1969wurzelbaume}, Halin's result that graphs without a (fat) subdivision of a countable clique have normal spanning trees  \cite{diestel2016simple,halin1978simplicial,pitz2021quickly,pitz2020unified} as well as results about forbidden minors for normal spanning trees
\cite{BowlerGeschkePitzNST,DiestelLeaderNST,halin2000miscellaneous,pitz2020d,pitz2021new}. For applications of normal trees that are not necessarily spanning, see~\cite{StarComb1StarsAndCombs,ApproximatingNormalTrees}.

\begin{lemma}\label{lem:maxNT}
Let $G$ be any connected graph and let $v\in G$ be any vertex.
Then there exists an inclusionwise maximal normal tree $T\subset G$ rooted at~$v$.
Every component of $G-T$ has infinite neighbourhood, and the down-closures in~$T$ of each such neighbourhood determines a normal ray of~$T$.
\end{lemma}
\begin{proof}
Inclusionwise maximal normal trees $T\subset G$ rooted at~$v$ exist by Zorn's lemma. By the defining property of a normal tree, the neighbourhood of every component of $G-V(T)$ forms a chain in the tree-order on $T$. We claim that all such neighbourhoods are in fact infinite, thus determining a unique normal ray of~$T$. Indeed, suppose for a contradiction that some component $C$ of $G-V(T)$ has finite neighbourhood. Let $t$ be maximal in the tree-order on $T$ amongst all neighbours of $C$. Then we can extend $T$ to a larger normal tree rooted in $v$ by adding any $t$--$C$ edge, a contradiction.
\end{proof}

\subsection{\texorpdfstring{$\boldsymbol{T}$}{T}-graphs}
A partially ordered set $(T,\le)$ is called an \emph{order tree} if it has a unique minimal element (called the \emph{root}) and all subsets of the form $\lceil t \rceil = \lceil t \rceil_T := \set{t' \in T}:{t'\le t}$
are well-ordered. Write  $\lfloor t \rfloor = \lfloor t \rfloor_T  := \set{t' \in T}:{t\le t'}$.
We abbreviate $\mathring{\dc{t}}:=\dc{t}\setminus\{t\}$ and $\underring{\uc{t}}:=\uc{t}\setminus\{t\}$.
For subsets $X\subset T$ we abbreviate $\dc{X}:=\bigcup_{t\in X}\dc{t}$ and $\uc{X}:=\bigcup_{t\in X}\uc{t}$.

A~maximal chain in~$T$ is called a \emph{branch}
of~$T$; note that every branch inherits a well-ordering from~$T$. 
The \emph{height} of~$T$ is the supremum of the order types of its branches. 
The \emph{height} of a point $t\in T$ is the order type of~$\mathring{\dc{t}}$.
The set $T^i$ of all points at height $i$ is
the $i$th \emph{level} of~$T$, and we
write $T^{<i} := \bigcup\set{T^j}:{j < i}$ as well as $T^{\leq i} := \bigcup\set{T^j}:{j \leq i}$.
If $t < t'$, we use the usual interval notation such as  $(t,t') = \{s \colon t< s  < t'\}$ for nodes between  $t$ and $t'$.
If there is no point between $t$ and~$t'$, we call $t'$ a \emph{successor}
of~$t$ and $t$ the \emph{predecessor} of~$t'$; if $t$ is not a successor
of any point it is called a \emph{limit}.
An order tree is \emph{special} if it can be partitioned into countably many antichains, and \emph{semi-special} if its successor nodes can be partitioned into countably many antichains.

A \emph{top} of a down-closed chain $\cC$ in an order tree $T$ is a limit $t\in T\setminus\cC$ with $\mathring{\dc{t}}=\cC$.
Note that a chain $\cC\subset T$ may have multiple tops.

An order tree $T$ is \emph{normal} in a graph $G$, if $V(G) = T$
and the two endvertices of any edge of $G$ are comparable in~$T$. We call $G$ a
\emph{$T$-graph} if $T$ is normal in $G$ and the set of lower neighbours of
any point $t$ is cofinal in $\mathring{\dc{t}}$. 
We mention the following standard results about $T$-graphs, and refer the reader to \cite[\S2]{brochet1994normal} for details.

\begin{lemma}
\label{lem_Tgraphproperties}
Let $(T,\leq)$ be an order tree and $G$ a $T$-graph.
\begin{enumerate}
\item \label{itemT1} For incomparable vertices $t,t'$ in $T$, the set $\downcl{t} \cap \downcl{t'}$ separates $t$ from $t'$ in $G$.
\item \label{itemT2} Every connected subgraph of $G$ has a unique $T$-minimal element.
\item \label{itemT3} If $T' \subset T$ is down-closed, the components of $G - T'$ are spanned by $\upcl{t}$ for $t$ minimal in $T\setminus T'$.
\item \label{itemT4} For any two vertices $t\le t'$ in~$T$, the interval $[t,t']$ is connected in~$G$.
\end{enumerate}
\end{lemma}

\section{Envelopes and concentrated sets of vertices}
\label{sec_envelopes}

An \emph{adhesion set} of a set of vertices or a subgraph $U \subset G$ is any subset of the form $N(C)$ for a component $C$ of $G-U$. The set or subgraph $U$ is said to have \emph{finite adhesion} in $G$ if all its adhesion sets are finite.

Let $G$ be a connected graph. An \emph{envelope} for a set of vertices $U \subset V(G)$ is a set of vertices $U^* \supseteq U$ of finite adhesion such that $\Abs{ U^*} = \Abs{U}$. We will show in \cref{thm:envelope} below that envelopes exist.

We need a preliminary lemma. 
A set of vertices $U$ is \emph{concentrated} (in its boundary $\Abs{U}$) if for every finite vertex set $X\subset V$ only finitely many vertices of~$U$ lie outside of $\bigcup_{\eps \in \Abs{U}} C(X,\eps)$. 
We say that a set of vertices $U$ is concentrated in an end $\eps$ to mean the case $\Abs{U} = \Set{\eps}$.

Recall that a subset $X$ of a poset $P=(P,{\le})$ is \emph{cofinal} in $P$, and ${\le}$, if for every $p\in P$ there is an $x\in X$ with $x\ge p$.
We say that a rooted tree $T\subset G$ contains a set $U$ \emph{cofinally} if $U\subset V(T)$ and $U$ is cofinal in the tree-order of~$T$. Note that such trees are easily constructed: any inclusion-minimal subtree $T$ of $G$ with $U \subset V(T)$ will contain $U$ cofinally---no matter which vertex of~$T$ we choose as root.

\begin{lemma}\label{lem:cofinalConcentrated}
Let $G$ be any graph, let $U\subset V(G)$ be a set of vertices, and suppose $T\subset G$ is a rooted tree that contains $U$ cofinally. Then the following assertions hold:
\begin{enumerate}
    \item $\Abs{T} = \Abs{U}$.
    \item If $U$ has finite adhesion, then so does~$T$.
    \item If $U$ is concentrated, then so is~$T$.
\end{enumerate}
\end{lemma}

\begin{proof}
Assertions (i) and (ii) already have been observed in \cite[Lemma~2.13]{StarComb1StarsAndCombs} and \cite[Lemma~4.3]{pitz2020d}. For convenience of the reader, we give a self-contained argument for all three properties (i)--(iii) at once. For this, suppose $T\subset G$ is a rooted tree that contains $U$ cofinally.

\begin{claim}
\label{claim_cofinally}
If $X$ is a finite set of vertices and $\cC$ a collection of components of $G-X$ such that $V(T)$ meets $\bigcup \cC$ infinitely, then so already does $U$.
\end{claim}

To see the claim, suppose for a contradiction that for some finite vertex set $X$ and some collection $\cC$ of components of $G-X$ we have $\bigcup \cC$ meets $U$ meets finitely, but $V(T)$ infinitely. By increasing $X$, we may assume that $\bigcup \cC \cap U = \emptyset$. Pick $t \in (T \setminus \downcl{X}) \cap \bigcup \cC$. Then $\upcl{t} \subseteq \bigcup \cC$ avoids $U$, a contradiction to the assumption that $U$ is cofinal in $T$.

Now to see (i), consider any end $\eps \notin \Abs{U}$. Then there is a finite set of vertices $X$ such that $C(X,\eps)$ avoids $U$. 
By \cref{claim_cofinally}, also $V(T)$ intersects $C(X,\eps)$ finitely, witnessing $\eps \notin \Abs{T}$. The argument for (iii) is similar. Finally, for (ii), consider a component $C'$ of $G-T$. Then there is a component $C$ of $G-U$ with $C' \subset C$. Assuming that $U$ has finite adhesion, $X=N(C)$ is finite.  \cref{claim_cofinally} applied to $X$ and $C$ yields that $V(T)$ intersects $C$ finitely. Then $N(C') \subset X \cup (V(T) \cap V(C))$ is finite, so $T$ has finite adhesion.
\end{proof}

\begin{theorem}
\label{thm:envelope}
Any \textnormal{[}concentrated\textnormal{]} set $U$ of vertices in a connected graph $G$ has a connected \textnormal{[}concentrated\textnormal{]} envelope $U^*$. Moreover, $U^*$ can be chosen such that $U^* \cap C$ is connected for every component $C$ of $G-U$.
\end{theorem}

\begin{proof}
Let $U$ be a given set of vertices in a connected graph $G$.
By Zorn's lemma, there is an inclusionwise maximal set of combs attached to $U$ with pairwise disjoint spines. Write $\cR$ for this collection of spines. Let $S$ be the set of all centres of (infinite) stars attached to $U$. We will show that 
$$U^* := U \cup \bigcup_{R \in \cR} V(R) \cup S$$
is an envelope for $U$. 
The verification relies on the following claim: 
\begin{claim}
\label{claim_1}
If $X$ is a finite set of vertices and $\cC$ a collection of components of $G-X$ such that $U^*$ meets $\bigcup \cC$ infinitely, then so already does $U$.
\end{claim}

To see the claim, consider some finite set of vertices $X$, and assume that $\cC$ is a collection of components of $G-X$ such that $U$ meets $\bigcup \cC$ finitely. Then $S$ avoids $\bigcup \cC$. 
Consider next a spine $R \in \cR$ that meets $\bigcup \cC$. Since $R$ eventually lies in a component $C \notin \cC$, we have that $R$ meets $\bigcup \cC$ finitely. Moreover, $R$ also meets $X$, and since the spines in~$\cR$ are pairwise disjoint, there are at most $|X|$ spines from $\cR$ that meet $\bigcup\cC$, and each does so finitely. Hence $\bigcup_{R \in \cR} V(R)$ meets $\bigcup \cC$ finitely, too, and the claim follows.

Now to see $\Abs{U^*} = \Abs{U}$, consider any end $\eps \notin \Abs{U}$. Then there is a finite set of vertices $X$ such that $C(X,\eps)$ avoids $U$. 
By \cref{claim_1}, also $U^*$ intersects $C(X,\eps)$ finitely, witnessing $\eps \notin \Abs{U^*}$. The argument that if $U$ is concentrated, then so is $U^*$, follows similarly from \cref{claim_1}.

To see that $U^*$ has finite adhesion, suppose for a contradiction that there is a component $C$ of $G-U^*$ with infinite neighbourhood.
Then by a routine application of \cref{StarCombLemma}, we either find a star or a comb attached to~$U^*$ whose centre $v$ or spine $R$ is contained in~$C$. Then for all finite sets of vertices $X$ (chosen disjoint from $v$ in the star case), the centre $v$ or a tail of $R$ lives in a component of $G-X$ that meets $U^*$ infinitely, and hence also $U$ by \cref{claim_1}. But then it is straightforward to inductively construct a star with centre $v$ or a comb with spine $R$ attached to $U$, violating the choice of $S$ or $\cR$ respectively.

To get a connected envelope for $U$ that also satisfies the moreover-part, note that there are inclusionwise minimal subtrees of $G$ containing $U^*$ which satisfy that $T \cap C$ is connected for every component $C$ of $G-U$: 
Take a spanning tree $T_C$ for every component $C$ of $G-U$, and extend the forest $\bigcup_C T_C$ to a rooted spanning tree $T$ of $G$. Then the down-closure of $U^*$ in $T$ can serve as the desired envelope by \cref{lem:cofinalConcentrated}. 
\end{proof}

For a concept that is somewhat stronger than our notion of envelopes (but whose existence proof is significantly more involved) see Polat's \emph{multiendings} from \cite{polat1996ends2}; a precursor to our notion of envelopes has been used in B\"urger and the first author's proof of \cite[Theorem~1]{StarComb2TheDominatedComb}.

 Recall that a sequence $(x_n)_{n \in \N}$ in a topological space $X$ is said to \emph{converge} to a subset $A \subset X$, denoted by $x_n \to A$ (as $n \to \infty$) if every open set containing $A$ contains also almost all $x_n$. Note that if $A= \Set{x}$ is a singleton, this reduces to the usual notion of convergence $x_n \to x$.
 
We will apply the following lemma for a single end only. However, for future reference we state it in its optimal form for a compact collection of ends.

\begin{lemma}
\label{lem_concentratedconvergence}
Let $U$ be a concentrated vertex set of $G$ of finite adhesion. 
Assume further that $\Abs{U}$ is compact.
Let $(\eps_n)_{n\in\N}$ be any sequence of ends in $\Omega(G) \setminus \Abs{U}$, and write $D_n$ for the unique component of $G-U$ in which $\eps_n$ lives.

Then $\eps_n \to \Abs{U}$ (as $n \to \infty$) 
if and only if the map $\N\ni n\mapsto N_G(D_n)$ is finite-to-one.
\end{lemma}

\begin{proof}
First, suppose that for some finite $X \subset U$ there are infinitely many $n\in \N$ with $N_G(D_n) = X$. Write $\cC$ for the collection of components of $G-X$ containing an end from $\Abs{U}$. Then $\Omega(X,\cC)$ is an open neighbourhood around $\Abs{U}$ avoiding infinitely many $\eps_n$, witnessing $\eps_n \not\to \Abs{U}$.

Now assume that the map $\N\ni n\mapsto N_G(D_n)$ is finite-to-one,
and let us suppose for a contradiction that the sequence $\eps_0,\eps_1,\ldots$ does not converge to~$\Abs{U}$ in the end space of~$G$.
Then there exists an open set around $\Abs{U}$ avoiding infinitely many $\eps_n$, and by compactness of $\Abs{U}$ we may find a finite vertex set~$X\subset V(G)$ 
such that infinitely many $\eps_n$ live outside of $\bigcup_{\eps \in \Abs{U}} C(X,\eps)$.
Since the map $n\mapsto D_n$ has finite fibres, we may assume without loss of generality that the finite vertex set~$X$ meets none of the components~$D_n$ in which these ends live.
Hence, infinitely many neighbourhoods $N_G(D_n)$ avoid~$\bigcup_{\eps \in \Abs{U}} C(X,\eps)$.
But each of these neighbourhoods consists of vertices in~$U$, and since the fibres of the map $n\mapsto N_G(D_n)$ are finite, the union of these neighbourhoods is an infinite subset of~$U$ that lies outside of~$\bigcup_{\eps \in \Abs{U}} C(X,\eps)$, contradicting the assumption that $U$ is concentrated.
\end{proof}

\section{Uniform \texorpdfstring{$T$}{T}-graphs and special order trees}
\label{sec_Sec3}
\label{sec_3}

\begin{definition}
A $T$-graph $G$ has
\begin{itemize}
    \item \emph{finite adhesion} if for every limit ordinal $\sigma$ all components of $G-T^{\le\sigma}$ have finite neighbourhoods;
\item \emph{uniformly finite adhesion}, or for short, is \emph{\tame }  if for every limit $t\in T$ there is a finite $S_t\subset\mathring{\dc{t}}$ such that every $t'>t$ has all its down-neighbours below $t$ inside $S_t$.
\end{itemize}
\end{definition}

The property for a $T$-graph to be \tame\ has been introduced by
Diestel and Leader in~\cite{DiestelLeaderNST}, and it~has proven useful in~\cite{Enddegree} as well. Our next two lemmas illuminate the relation between finite and uniformly finite adhesion.

\begin{lemma}\label{finiteAdhesionStrenghtening}
The following are equivalent for a $T$-graph $G$:
\begin{enumerate}
    \item $G$ has finite adhesion.
    \item For every ordinal $\sigma$ all components of $G-T^{\le\sigma}$ have finite neighbourhoods.
    \item Every vertex set of the form~$\upcl{t}$ for a successor $t\in T$ has finite neighbourhood in~$G$.
\end{enumerate}
\end{lemma}
\begin{proof}
(i) $\Rightarrow$ (ii). For limits $\sigma$, this holds by assumption, so suppose that $\sigma$ is a non-limit and consider any component~$C$ of $G-T^{\le\sigma}$.
Let us assume for a contradiction that $C$ has infinitely many neighbours.
Using that these form a well-ordered chain and that $\sigma$ is a non-limit, we find an $\omega$-chain of infinitely many neighbours $t_0<t_1<\cdots$ of~$C$ and a limit $t\in T^{<\sigma}$ such that $t_n<t$ for all~$n<\omega$.
Let $\mu$ denote the height of~$t$ and note that $\mu$ is a limit.
Then the component of $G-T^{\le\mu}$ which contains~$C$ has infinite neighbourhood, contradicting our assumption that $G$ has finite adhesion.

(ii) $\Rightarrow$ (iii). Suppose $t$ is a successor, and let $\sigma$ be the height of the predecessor of $t$. By \cref{lem_Tgraphproperties}~\ref{itemT3}, the unique component of $G-T^{\leq \sigma}$ containing $t$ is spanned by $\upcl{t}$, so has finite neighbourhood.

(iii) $\Rightarrow$ (i). Let $\sigma$ be a limit ordinal. By \cref{lem_Tgraphproperties}~\ref{itemT3}, the components of $G-T^{\leq \sigma}$ are spanned  by $\upcl{t}$ for $t \in T^{\sigma+1}$, so have finite neighbourhoods by assumption (iii).
\end{proof}

\begin{lemma}\label{uniformfiniteAdhesionStrenghtening}
The following are equivalent for a $T$-graph $G$:
\begin{enumerate}
    \item $G$ has uniformly finite adhesion.
    \item Every vertex set of the form~$\underring{\upcl{t}}$ for a limit $t\in T$ has finite neighbourhood in~$G$.
    \item Every vertex set of the form~$\underring{\upcl{t}}$ has finite neighbourhood in~$G$.
\end{enumerate}
\end{lemma}

\begin{proof}
(i) $\Rightarrow$ (ii). If $G$ has uniformly finite adhesion, then every vertex set of the form~$\underring{\upcl{t}}$ for a limit $t\in T$ has all its neighbours in the finite set $S_t\cup\{t\}$.

(ii) $\Rightarrow$ (iii). Let us assume for a contradiction that some $\underring{\upcl{t}}$ has infinitely many neighbours. 
Then we find an $\omega$-chain of infinitely many neighbours $t_0<t_1<\cdots$ of~$\underring{\upcl{t}}$ and a limit $\ell\le t$ such that $t_n<\ell$ for all~$n<\omega$. But then also $\underring{\upcl{\ell}} \supset \underring{\upcl{t}}$ has infinite neighbourhood, a contradiction.

(iii) $\Rightarrow$ (i). Let $t$ be a limit node of $T$. Then $S_t = N(\underring{\upcl{t}}) \setminus \{t\}$ is as desired.
\end{proof}

Comparing \cref{finiteAdhesionStrenghtening}~(iii) with \cref{uniformfiniteAdhesionStrenghtening}~(iii) explains the name `uniform' and in particular shows that uniformly finite adhesion implies finite adhesion.

\begin{lemma}
\label{lem_downNeighboursOmegaChain}
If $G$ is a $T$-graph of finite adhesion and $t\in T$ is a limit, then the down-neighbours of~$t$ form a cofinal $\omega$-chain in~$\mathring{\dc{t}}$.
\end{lemma}
\begin{proof}
The down-neighbours of~$t$ form an infinite chain~$\cC\subset\mathring{\dc{t}}$.
Since $G$ has finite adhesion, below every successor in~$\mathring{\dc{t}}$ there are only finitely many elements of~$\cC$.
It follows that $\cC$ must be an $\omega$-chain.
\end{proof}

Our next result characterises on which order trees there exist $T$-graphs of (uniformly) finite adhesion. 
Recall from~\cite{baumgartner1970results} that a partial order $P$ is said to \emph{embed} into a partial order $Q$ if there is a map $f \colon P \to Q$ such that $p < p'$ implies $f(p) < f(p')$ (in particular, such maps are not required to be injective).

\begin{theorem}[Baumgartner and Galvin {\cite[Section~4.1]{baumgartner1970results}}]\label{specialVsEmbedding}
    Let $(T,\leq)$ be an order tree.
    \begin{enumerate}
        \item $(T,\le)$ is special if and only if  $(T,\leq)$ embeds into $(\mathbb{Q},\leq)$.
        \item $(T,\le)$ is semi-special if and only if $(T,\leq)$ embeds into $(\mathbb{R},\leq)$.
\end{enumerate}
\end{theorem}

\begin{theorem}
\label{prop_specialtrees}
Let $(T,\leq)$ be an order tree.
\begin{enumerate}
\item There exists a $T$-graph of uniformly finite adhesion if and only if $(T,\leq)$ is special.
\item There exists a $T$-graph of finite adhesion if and only if $(T,\leq)$ is semi-special.
\end{enumerate}
\end{theorem}

\begin{proof}
(i) 
Suppose first that $T$ is special with antichain partition $\{U_n\colon n \in \N\}$. Following Diestel, Leader and Todorcevic~\cite{DiestelLeaderNST},  we construct a $T$-graph as follows: Successors are connected to their predecessors, and given a limit $t \in T$, pick down-neighbours $t_0 <_T t_1 <_T t_2 <_T \cdots<_T  t$ with $t_i \in U_{n_i}$ recursively such that $t_{i-1} < t_i < t$ and each $n_{i}$ is smallest  possible. We claim that the resulting $T$-graph $G$ has uniformly finite adhesion.
Indeed, $S_t$ may be chosen to consist of all elements of $\mathring{\lceil t \rceil}$ contained in an antichain with smaller index than the antichain containing $t$.

Conversely, suppose that there is a $T$-graph $G$ of uniformly finite adhesion with root~$r$, say. We argue that there exists a regressive function $f \colon (T - r) \to T$ such that each fibre can be covered by countably many antichains (then $T$ itself is special by a well-known observation of Todorcevic \cite[Theorem~2.4]{Todorcevic1981}). Indeed, map all successors to their unique predecessor, and map a limit $t$ to any of its neighbours in $(\max S_t,t)$. This regressive map is ${\leq}2-1$ on branches of $T$: Indeed, suppose for a contradiction that there are two limits $t<t'$ with $f(t) = f(t')$. Since $t' > t > f(t) = f(t')$ we have $f(t') \in S_t$. But then $f(t) \in S_t$, a contradiction.

(ii) Suppose that the successors of $T$ admit an antichain partition $(U_n)_{n \in \N}$. Following the same procedure as above, we construct a $T$-graph as follows: Successors are connected to their predecessors, and given a limit $t \in T$, pick down-neighbours $t_0 <_T t_1 <_T t_2 <_T \cdots<_T  t$ with $t_i \in U_{n_i}$ recursively such that $t_{i-1} < t_i < t$ and each $n_{i}$ is smallest  possible. We claim that the resulting $T$-graph $G$ has finite adhesion. 
Indeed, for any limit ordinal $\sigma$ and any component $C$ of $G-T^{\le\sigma}$ we have $C = \upcl{s}$ for a unique successor node $s$ of some limit $t \in T^\sigma$ by \cref{lem_Tgraphproperties}~\ref{itemT3}. Hence, $N(C)$ is finite, as all elements of $N(C) \cap \mathring{\dc{t}}$ belong to  antichains with smaller indices than the antichain containing $s$.

Conversely, suppose that there is a $T$-graph $G$ of finite adhesion. 
Let $L\subset T$ consist of all  limits of~$T$ that do have successors in $T$, and for every limit $\ell\in L$ let the set $S(\ell)$ consist of all the successors of $\ell$ in~$T$. 
We obtain a new order tree $T'$ from~$T$ as follows:
First, add for each $\ell\in L$ and $s\in S(\ell)$ a new node $v(\ell,s)$ that we declare to be a successor of $\ell$ and a predecessor of~$s$. Then delete~$L$.
Note that there is a natural epimorphism $\varphi\colon T' \to T$ which is the identity on $T\setminus L$ and which sends each $v(\ell,s)$ to $\ell$.

Let $G'$ be the $T'$-graph with $V(G'):=T'$ and 
$$E(G') := \set{tt'}:{t < t' \in T' \; \wedge \; \varphi(t)\varphi(t') \in E(G)}.$$
Consider a limit $t=v(\ell,s)$ of $T'$. Since $G$ has finite adhesion, $\upcl{s}_T$ has finite neighbourhood say $S_s$ in~$G$. Since every $t' > t$ in $T'$ satisfies $\varphi(t') \in \upcl{s}_T$, the set $S_t=\varphi^{-1}(S_s) \cap \dc{t}_{T'}$ witnesses that $G'$ has uniformly finite adhesion.
It follows from (i) that $T'$ is special, so admits a countable antichain partition. This induces an antichain partition on the successors of~$T'$.
\end{proof}

\begin{corollary}
\label{cor_countablebranches}
If $G$ is a $T$-graph of finite adhesion, then all branches of~$T$ are countable. 
\end{corollary}
\begin{proof}
An uncountable branch 
does not embed into the reals.
\end{proof}

\begin{corollary}
\label{cor_unctblclique}
$T$-graphs of finite adhesion do not contain uncountable clique minors.
\end{corollary}

\begin{proof}
Suppose some $T$-graph of finite adhesion contains an uncountable clique minor. Then by a result due to Jung \cite{jung1967zusammenzuge}, this graph would also contain a subdivision of an uncountable clique. By \cref{cor_countablebranches}, some pair of branch vertices $v,w$ of this clique must be incomparable in $T$. By \cref{lem_Tgraphproperties}~\ref{itemT1} the set $\lceil v \rceil \cap \lceil w \rceil$ is a countable separator between $v$ and $w$, a contradiction.
\end{proof}

\section{End spaces of \texorpdfstring{$T$}{T}-graphs of finite adhesion}
\label{sec_4}

Recall from the introduction that a down-closed chain $\cC$ in an order tree $T$ is called a \emph{high-ray} of~$T$ if its order-type 
is a limit ordinal (equivalently: if $\cC$ has no maximal element).
We denote the set of all high-rays of~$T$ by~$\HR(T)$.
In this section, we show that the end space of a $T$-graph of finite adhesion can be understood through the high-rays of~$T$. Note that in this set-up, all high-rays in $T$ have cofinality $\omega$ due to \cref{cor_countablebranches}.

For this, let $G$ be any $T$-graph of finite adhesion.
Every $t\in T$ induces a bipartition $\{\,\upcl{t}\,,\,T\setminus\uc{t}\,\}$ of~$T$.
Since $T$ is the vertex set of~$G$, it follows that $\Omega(G)=\Abs{\upcl{t}}\cup\Abs{(T\setminus\upcl{t})}$.
For every end~$\eps$ of~$G$, let us put
\[
    \sigma(\eps):=\{\,t\in T\mid \eps\in\Abs{\upcl{t}}\setminus\Abs{(T\setminus\upcl{t})}\,\}.
\]
Clearly, $\sigma(\eps)$ is a down-closed chain in~$T$.
At the end of the next section, we will see that

\begin{lemma}\label{SigmaBijection}
The map $\sigma$ is a bijection between the end space $\Omega(G)$ and the set $\HR(T)$ of all high-rays of~$T$, for each $T$-graph~$G$ of finite adhesion.
\end{lemma}

\begin{corollary}
\label{cor_ctbldegree}
All ends of $T$-graphs of finite adhesion have countable degree.
\end{corollary}

\begin{proof}
Let $G$ be a $T$-graph of finite adhesion, and let $\eps$ be an end of~$G$.
Put $\varrho:=\sigma(\eps)$, and assume for a contradiction that $\{\,R_i\mid i<\aleph_1\,\}$ is a collection of uncountably many pairwise disjoint rays in the end~$\eps$.
Since $\varrho$ has cofinality~$\omega$ and $G$ has finite adhesion, it follows that all but countably many rays~$R_i$ are included in~$G[\,\uc{U}\,]$ where $U$ is the set of all tops of~$\varrho$.
Pick any ray $R_i\subset G[\,\uc{U}\,]$.
Then $R_i$ has a unique $T$-minimal vertex by \cref{lem_Tgraphproperties}~\ref{itemT2}, so~$R_i$ has a tail~$R_i'$ which avoids~$U$.
In particular, there is a vertex $s\in T$ that is a successor of a top of~$\varrho$ and satisfies $R_i'\subset G[\,\upcl{s}\,]$.
But then the finite neighbourhood $N_G(\upcl{s})$ separates~$R_i'$ from~$\eps$, a contradiction.
\end{proof}

We now characterize convergent sequences $
\eps_n \to \eps$ for ends of a \tame\ $T$-graph in terms of combinatorial behaviour of their corresponding high rays in $T$.

\begin{lemma}
\label{lem_highraytop}
Let $G$ be a \tame\ $T$-graph. Let $\eps$ and $\eps_n$ $(n \in \N)$ be ends of $G$, and let $\varrho:=\sigma(\eps)$ and $\varrho_n:=\sigma(\eps_n)$ be the corresponding high-rays in $T$. Let $A\subset\N$ consist of all numbers~$n$ for which $\varrho \subsetneq \varrho_n$, and let $B:=\N\setminus A$.
\begin{enumerate}
    \item We have convergence $\eps_n \to \eps$ in~$\Omega(G)$ for $n \in A$ and $n\to\infty$ if and only if $A$ is infinite and for every top $t$ of $\varrho$ there are only finitely many $n \in A$ with $t \in \varrho_n$.
    \item We have convergence $\eps_n \to \eps$ in~$\Omega(G)$ for $n \in B$ and $n\to\infty$ if and only if $B$ is infinite and for every successor node $t \in \varrho$ there are only finitely many $n \in B$ with $\varrho \cap \varrho_n \subset \mathring{\dc{t}}$.
\end{enumerate}
\end{lemma}

\begin{proof}
(i) We show the forward implication indirectly.
For this, suppose that there is a top $t$ of~$\varrho$ such that $t\in\varrho_n$ for infinitely many~$n\in A$.
Then the finite vertex set $S_t\cup\{t\}$ separates~$\eps$ from infinitely many~$\eps_n$ at once, a contradiction.

For the backward implication, let $X\subset V(G)$ be any finite vertex set; we have to find a number $N\in\N$ such that $\eps_n$ lives in $C(X,\eps)$ for all~$n\in A$ with $n\ge N$.
Let $Y$ be the set of all tops $t$ of~$\varrho$ for which $\uc{t}$ meets~$X$.
By assumption, we find a large enough $N\in\N$ such that all high-rays $\varrho_n$ with $n\in A$ and $n\ge N$ avoid~$Y$.
We claim that all corresponding ends $\eps_n$ live in $C(X,\eps)$. For this, let any $n\in A$ with $n\ge N$ be given.
Let $t'_0<t'_1<\cdots$ be a cofinal $\omega$-chain in the high-ray $\varrho_n$ such that~$t'_0$ is a top of~$\varrho$.
Let $t_0<t_1<\cdots$ be a cofinal $\omega$-chain in the high-ray $\varrho$ such that $X$ has no vertex on~$\varrho$ above or equal to~$t_0$ and $t_0$ is a neighbour of~$t'_0$.
Using that the intervals $[t_k,t_{k+1}]$ and $[t'_k,t'_{k+1}]$ of the $T$-graph $G$ are connected for all $k\in\N$ by \cref{lem_Tgraphproperties}~\ref{itemT4}, we find rays $R\in\eps$ and $R_n\in\eps_n$ which start in $t_0$ and $t'_0$, respectively, and avoid~$X$.
Since $t_0 t'_0$ is an edge of~$G$, it follows that $\eps_n$ lives in~$C(X,\eps)$.

(ii) We show the forward implication indirectly.
For this, suppose that there is a successor node $t\in \varrho$ such that~$\varrho\cap\varrho_n\subset\mathring{\dc{t}}$ for infinitely many~$n\in B$.
Then $N_G(\uc{t})$ is a finite vertex set by \cref{finiteAdhesionStrenghtening}, and it separates infinitely many~$\eps_n$ from~$\eps$ simultaneously.

For the backward implication, let any finite vertex set $X\subset V(G)$ be given.
Let $t\in\varrho$ be any successor such that $X\subset V_{\mathring{\dc{t}}}$.
By assumption, we find a large enough number $N\in B$ such that $t\in\varrho_n$ for all~$n\ge N$ (with $n\in B$).
Let $t_0<t_1<\cdots$ be a cofinal $\omega$-chain in the high-ray~$\varrho$ with $t_0:=t$.
Since the intervals $[t_n,t_{n+1}]\subset\varrho$ are connected in~$G$ by \cref{lem_Tgraphproperties}~\ref{itemT4}, we find a ray $R\in\eps$ in $G-X$ which starts in~$t$.
Similarly, we find a ray $R_n\in\eps_n$ in $G-X$ which starts in~$t$, for every~$n\ge N$.
Hence $\eps_n\in\Omega(X,\eps)$ for all $n\ge N$.
\end{proof}

We now show that for \tame\ $T$-graphs, the topology described in \cref{lem_highraytop} is the topology of $\cR(T)$. The definition of $\cR(T)$ as a subspace of $\Set{0,1}^T$ in the introduction is equivalent to the assertion that the topology of $\cR(T)$ is the smallest topology in which all sets of the form $[t] = \set{x \in \cP(T)}:{t \in x}$ and their complements $[t]^\complement = \set{x \in \cP(T)}:{t \notin x}$ are open. 

\begin{lemma}
\label{lem_standardbase}
	For ray spaces, a local neighbourhood base $\cC(x)$ at some high-ray $x \in \cR(T)$ is given by sets of the form
$$[t,F] := [t] \setminus \bigcup_{s \in F} [s] = [t] \setminus [F] $$
where $t \in x$ and $F$ is a finite set of tops of $x$. 
\end{lemma}
\begin{proof}
	By definition of our subbase, every open set $U$ with $x \in U \subseteq \cR(T)$ contains a basic set $B$ of the form
$$x \in B = \bigcap_{t \in E} [t] \cap \bigcap_{s \in F} [s]^\complement= \bigcap_{t \in E} [t] \setminus \bigcup_{s \in F} [s] \subseteq U$$
for some finite sets $E,F \subset T$. Pick $B$ with $|E| + |F|$ of minimal size. Since $x \in B$ implies $E \subset x$, we could replace $E$ by its maximum. Hence, $|E| \leq 1$. 
Next, we claim that for every $s\in F$ there is a top $s'$ of  $x$ with $s' \leq s$. Indeed, if say $\dc{s_0} \cap x \subsetneq x$, pick $t \in x \setminus (E \cup \dc{s_0})$ to obtain
$$x \in [t] \setminus \bigcup_{s \in F_0} [s] \subseteq B$$
for $F_0 = F \setminus \Set{s_0}$, contracting the minimality of $|E|+|F|$. Hence, for every $s\in F$ there is a top $s'$ of  $x$ with $s' \leq s$, and by setting $F' = \set{s'}:{s \in F}$, we obtain
\[x \in [t,F'] \subseteq B \subseteq U \subseteq \cR(T) \]
with $t \in x$ and $F'$ a finite set of tops of $x$ as desired.
\end{proof}

The following result yields the equivalence $(2) \Leftrightarrow (3)$ in our Representation Theorem~\ref{thm_rep}.
 
\begin{proposition}
\label{prop_specialendspaces}
The endspace of a \tame\ $T$-graph is naturally homeomorphic to $\cR(T)$.
\end{proposition}

\begin{proof}
Let $G$ be a \tame\ $T$-graph. Then $T$ is special by \cref{prop_specialtrees}. By \cref{lem_highraytop}, the end space $\Omega(G)$ is naturally homeomorphic to the topology $\tau_{seq}$ on $\cR(T)$ that is given by describing the convergent sequences as follows:
Let $x$ and $x_n$ ($n \in \N$) be high-rays in a special order tree~$T$. Let $A\subset\N$ consist of all numbers~$n$ for which $x \subsetneq x_n$, and let $B:=\N\setminus A$.
\begin{enumerate}
    \item We have convergence $x_n \to x$  for $n \in A$ and $n\to\infty$ if and only if $A$ is infinite and for every top $t$ of $x$ there are only finitely many $n \in A$ with $t \in x_n$.
    \item We have convergence $x_n \to x$  for $n \in B$ and $n\to\infty$ if and only if $B$ is infinite and for every node $t \in x$ there are only finitely many $n \in B$ with $x \cap x_n \subset \mathring{\dc{t}}$.
\end{enumerate}
Write $\tau$ for the standard topology on the ray space $\cR(T) \subseteq \Set{0,1}^T$. In order to verify that $\tau$ and $\tau_{seq}$ induce the same closed sets, given $X \subseteq \cR(T)$ we need to show that $x \in \closure{X}$ with respect to $\tau$ if and only if there is a sequence $(x_n) \subseteq X$ such that $x_n \to x$ in $\tau_{seq}$. 

For the forwards implication, suppose that $x \in \closure{X} \setminus X$. If there are infinitely many tops $s_n$ ($n \in \N$) of $x$ such that $[s_n] \cap X \neq \emptyset$ then select $x_n \in [s_n] \in X$ and note that $x_n \to x$ in $\tau_{seq}$ according to (1). Hence, we may suppose that there are only finitely many tops $F = \Set{s_1,\ldots, s_k}$ of $x$ such that $X \cap [s_i] \neq \emptyset$. Let $X' = X \setminus [F]$. Since $[F] \not\ni x$ is clopen, it follows that $x \in \closure{X'}$. Let $t_n$ be an increasing, cofinal sequence in $x$ (which exists since $T$ is special). 
Since $\bigcap_{n \in \N}[t_n] \cap X' = \emptyset$ but $[t_n] \cap X' \neq \emptyset$, we can choose pairwise distinct $x_n \in X' \cap [t_n]$, and we have convergence $x_n \to x$ in $\tau_{seq}$ according to (2).

For the backwards implication, suppose $x_n \to x$ in $\tau_{seq}$. We show that $x \in \closure{X}$ with respect to~$\tau$. Let $[t,F]$ be a standard basic open neighbourhood of $x$ where $t \in x$ and $F$ a finite set of tops of~$x$. If $x_n \to x$ according to (1), then only finitely many members of $\set{x_n}:{n \in \N}$ lie above $F$, so infinitely many $x_n$ belong to $[t,F]$.  If $x_n \to x$ according to (2), then only finitely many members of $x_n$ do not contain $t$, so infinitely many $x_n$ belong to $[t,F]$. 
\end{proof}

\begin{prob}
Find a combinatorial characterisation when two special trees represent the same end space.
\end{prob}

\section{High-rays and \pt s that display ends}\label{sec:highrays}
\label{sec_5}

Inspired by the normal partition trees from \cite{brochet1994normal}, in this section we introduce our main structuring tool called `partition tree', and discuss how it helps in capturing the end space of a graph.

If $\cV=(\,V_t\colon t\in T\,)$ is a partition of~$V(G)$ whose classes are indexed by the points of an order tree~$T$, then for every subset $T'\subset T$ we write $V_{T'}:=\bigcup\,\{\,V_t\mid t\in T'\,\}$.
Sometimes, when $T'$ is given by a long formula, we will write $V\rest T':=V_{T'}$ instead.

\begin{definition}\label{def_partition_tree}
A \emph{\pt } of~$G$ is a pair $(T,\cV)$ where $T$ is an order tree and $\cV=(\,V_t\colon t\in T\,)$ is a partition of~$V(G)$ into connected vertex sets~$V_t$ (also called \emph{parts}) such that the following conditions hold:
\begin{enumerate}[label=(PT\arabic*)]
    \item\label{pt1} The contraction minor $\dot{G}:=G/\cV$ (with all arising parallel edges and loops deleted) is a $T$-graph.
    \item\label{pt2} $(T,\cV)$ has \emph{finite adhesion} in that $N_G(V_{\uc{t}})$ is finite for all successors $t\in T$.
    \item\label{pt3} All parts $V_t$ with $t$ not a limit are required to be singletons.
\end{enumerate}
\end{definition}
The first condition implies that only connected graphs can have \pt s.
The second condition implies by \cref{finiteAdhesionStrenghtening}~(iii) that the $T$-graph $\dot{G}$ has finite adhesion, but note that the converse is false: if $t\in T$ is a successor of a limit~$\ell$, then~\ref{pt2} forbids that the sole vertex in~$V_t$ (cf.~\ref{pt1}) has infinitely many neighbours in~$V_\ell$, but in~$\dot{G}$ all these neighbours would collapse to a single vertex. However, if $G$ is already a $T$-graph of finite adhesion, then $T$ defines a \pt\ $(T,\cV)$ of~$G$ if we let $V_t:=\{t\}$ for all~$t\in T$.

\begin{lemma}\label{lem_npt_countablebranches}
If $(T,\cV)$ is a \pt\ of a graph~$G$, then all branches of $T$ are countable.
\end{lemma}
\begin{proof}
Since $(T,\cV)$ has finite adhesion, so does the $T$-graph $\dot{G}$, and we may apply \cref{cor_countablebranches} to $\dot{G}$.
\end{proof}

Recall more generally from the previous section that we already know exactly how the end space of the $T$-graph $\dot G$ looks like. 
We now investigate the relationship between the ends of $G$ and the ends of $\dot{G}$.

Let $G$ be any graph, and let $(T,\cV)$ be any \pt\ of~$G$.
Each $t\in T$ induces a bipartition $\{\,V_{\uc{t}},V_{T\setminus\uc{t}}\,\}$ of~$V(G)$, hence $\Omega(G)=\Abs{V_{\uc{t}}}\cup \Abs{V_{T\setminus\uc{t}}}$
follows for all~$t$.
For every end $\eps$ of~$G$ let us put
\[
    O(\eps):=\{\,t\in T\mid \eps\in\Abs{V_{\uc{t}}}\setminus \Abs{V_{T\setminus\uc{t}}}\,\}.
\]
Clearly, $O(\eps)$ is a non-empty down-closed chain in~$T$, and it is countable because all branches of~$T$ are countable by \cref{lem_npt_countablebranches}.
If $O(\eps)$ has no maximal element, then~$O(\eps)$ is a high-ray of $T$ and we say that $\eps$ \emph{corresponds} to this high-ray.
Otherwise $O(\eps)$ has a maximal element~$t$. Then we say that $\eps$ \emph{lives at}~$t$ and \emph{in} the part~$V_t$.

\begin{lemma}\label{lem:livingInPoints}
If $O(\eps)$ has a maximal element~$t$, then $t$ must be a limit such that every ray in~$\eps$ has a tail in~$G[V_{\uc{t}}]$ and meets~$V_t$ infinitely often.
\end{lemma}
\begin{proof}
Since $t$ is contained in~$O(\eps)$, the end~$\eps$ does not lie in the closure of~$V_{T\setminus\uc{t}}$, i.e., every ray in~$\eps$ has a tail in~$G[V_{\uc{t}}]$.
Let $s_i$ ($i\in I$) be the successors of~$t$ in~$T$.
Since no~$s_i$ is contained in~$O(\eps)$ and $(T,\cV)$ has finite adhesion, every ray in~$\eps$ meets each vertex set $V_{\uc{s_i}}$ only finitely often.
As $V_t$ pairwise separates all vertex sets $V_{\uc{s_i}}$ in the subgraph~$G[V_{\uc{t}}]$ which contains a tail of every ray in~$\eps$, it follows that every ray in~$\eps$ meets~$V_t$ infinitely often.
In particular, $V_t$ is infinite, so $t$ must be a limit.
\end{proof}

Consider the map $\tau\colon\Omega(G)\to \HR(T)\sqcup T$ that takes each end of~$G$ to the high-ray or point of~$T$ which it corresponds to or lives at, respectively.
We say that $(T,\cV)$ \emph{displays} a set $\Psi$ of ends of~$G$ if $\tau$ restricts to a bijection $\tau\rest\Psi\colon\Psi\to\HR(T)$ between $\Psi$ and the high-rays of~$T$ and maps every end that is not contained in~$\Psi$ to some point of~$T$.

\begin{lemma}\label{lem:findRayForHighRay}
Let $(T,\cV)$ be a \pt\ of~$G$ with a high-ray $\varrho\subset T$, let $u_0,u_1,\ldots$ be vertices of~$G$, and let $t_0<t_1<\cdots$ be a cofinal $\omega$-chain in~$\varrho$ such that $u_n\in V_{t_n}$ for all $n\in\N$.
Then $G[\,V\rest\{t\in\varrho\colon t\ge t_0\}\,]$ contains a comb attached to~$\{\,u_n\mid n\in\N\,\}$.
Moreover, the spine of this comb belongs to an end of~$G$ which corresponds to~$\varrho$.
\end{lemma}
\begin{proof}
Since $\dot{G}$ is a $T$-graph, each interval $[t_n,t_{n+1}]$ is connected in it by \cref{lem_Tgraphproperties}~\ref{itemT4}, so in particular it contains a ray $R$ with $V(R)\subset\varrho$ that traverses all $t_n$ in the correct order.
Using that the induced subgraphs $G[V_t]$ with $t\in T$ are connected, it is straightforward to construct the desired comb from~$R$. 
\end{proof}

\begin{lemma}\label{lem:uniqueLimit}
Let $(T,\cV)$ be a \pt\ of~$G$, and let $\eps$ be any end of~$G$ that corresponds to a high-ray~$\varrho\subset T$.
Then
\[\bigcap_{t\in\varrho}\Abs\,\big(\,\medcup\,\{\,V_s\mid s\in\varrho\text{ and }s\ge t\,\}\,\big)=\{\eps\}.\]
In particular, no other end of~$G$ corresponds to~$\varrho$.
\end{lemma}
\begin{proof}
First, we show that $\eps$ is contained in the intersection on the left side of the equation.
For this, let any $t\in\varrho$ be given and put $X:=\bigcup\,\{\,V_s\mid s\in\varrho\text{ and }s\ge t\,\}$.
Assume for a contradiction that $\eps$ is not contained in~$\Abs{X}$.
Then $\eps$ contains a ray $R$ which avoids~$X$. Since $t \in \varrho$ implies $\eps \notin \Abs{V_{T\setminus\uc{t}}}$, a tail (so without loss of generality all) of $R$ avoids $V_\varrho$. 
By \cref{lem_Tgraphproperties}~\ref{itemT3}, the components of $G-V_{\varrho}$ are of the form $V_{\upcl{r}}$ for $r$ minimal in $T-\varrho$, and since $\eps$ corresponds to~$\varrho$, the component of $G-V_{\varrho}$ containing $R$ has to be of the form $G[V_{\upcl{r}}]$ for $r$ a top of $\varrho$.

This last statement implies $\eps\in\Abs{V_{\upcl{r}}}$. 
By definition of $O(\eps) = \varrho$, the fact $r\notin\varrho$ implies $\eps\in\Abs{V_{T\setminus \upcl{r}}}$, and so we can extend $R$ to a comb in~$G$ which is attached to $V_{T\setminus \upcl{r}}$.
However, since $\dot{G}$ is a $T$-graph and $(T,\cV)$ has finite adhesion, it follows 
that all but finitely many of the paths that have been added to~$R$ in order to obtain this comb must pass through~$X$, which gives $\eps\in\Abs{X}$, a contradiction.

It remains to show the forward inclusion of the equation.
For this, let $\delta$ be an element of the intersection on the left.
We have just shown that~$\eps$ is an element of this intersection as well.
Pick rays $R\in\eps$ and $S\in\delta$.
We show that $R$ and $S$ are equivalent in~$G$.
Let $t_0<t_1<\cdots$ be a cofinal $\omega$-chain in the high-ray~$\varrho$, and put $X_n:=\bigcup\,\{\,V_s\mid s\in\varrho,s\ge t_n\,\}$.
We extend both $R$ and $S$ to combs $C_R$ and $C_S$ in~$G$ with teeth $u_R^n$ and $u_S^n$ in $X_n$ for each $n\in\N$, respectively.
Using that all intervals $[t_n,t_{n+1}]$ of the $T$-graph~$\dot{G}$ are connected by \cref{lem_Tgraphproperties}~\ref{itemT4}, we find infinitely many pairwise disjoint $\{\,u_R^n\mid n\in\N\,\}$--$\{\,u_S^n\mid n\in\N\,\}$ paths in~$G$, which shows that $R$ and $S$ are equivalent as desired.
\end{proof}

\begin{lemma}\label{lem:displays}
Every \pt\ $(T,\cV)$ of~$G$ displays the ends of~$G$ that correspond to the high-rays of~$T$.
\end{lemma}
\begin{proof}
Every high-ray of~$T$ has some ends of~$G$ corresponding to it by \cref{lem:findRayForHighRay}.
Every high-ray of~$T$ has at most one end of~$G$ corresponding to it by \cref{lem:uniqueLimit}.
\end{proof}

\begin{lemma}\label{lem:PTtopNbhdClosure}
Let $(T,\cV)$ be a \pt\ of~$G$ and $t\in T$ a limit.
Then $N(V_{\uc{t}})$ is concentrated in the unique end of~$G$ that corresponds to the high-ray~$\mathring{\dc{t}}$.
\end{lemma}
\begin{proof}
Using \cref{lem:displays} we may let $\eps$ be the unique end of~$G$ that corresponds to the high-ray~$\mathring{\dc{t}}$.
To show that $N(V_{\uc{t}})$ is concentrated in~$\eps$, it suffices to show that for every infinite set~$U$ of neighbours of $V_{\uc{t}}$ there exists a comb in~$G$ attached to~$U$ with its spine in~$\eps$.
For this, let~$U$ be any infinite set of neighbours of~$V_{\uc{t}}$.
Since $\mathring{\dc{t}}$ is a high-ray, it contains a cofinal $\omega$-chain $t_0<t_1<\cdots$, and since $(T,\cV)$ has finite adhesion, 
by \cref{lem_downNeighboursOmegaChain} all but finitely many of the vertices of~$U$ are contained in~$V\rest\{t'\in T\mid t_n\le t'<t\}$ for each $n\in\N$.
Hence we find a cofinal $\omega$-chain $t'_0<t'_1<\cdots$ in~$\mathring{\dc{t}}$ such that $V_{t'_n}$ contains a vertex~$u_n$ of~$U$ for all~$n\in\N$.
Applying \cref{lem:findRayForHighRay} to~$\{\,u_n\mid n\in\N\,\}$ yields the desired comb.
\end{proof}

Suppose now that $T$ is any order tree and that $G$ is a $T$-graph of finite adhesion.
Then $T$ defines a \pt\ $(T,\cV)$ of~$G$ if we let $V_t:=\{t\}$ for all~$t\in T$.
By \cref{lem:livingInPoints}, no end of~$G$ can live in a part of this \pt , so every end of~$G$ corresponds to a high-ray of~$T$ via~$\tau$.
Then $(T,\cV)$ displays all the ends of~$G$ by \cref{lem:displays}.
To distinguish this special case from the general case, we denote the bijection $\Omega(G)\to\HR(T)$ by~$\sigma$ in this case.
This proves \cref{SigmaBijection}.

\section{\Ff\ \pt s that display all ends}

\label{sec_6}

For any \tame\ $T$-graph~$G$, \cref{lem_highraytop} allows us to fully understand the topology of~$\Omega(G)$ through the combinatorial behaviour of high-rays in~$T$.
In this section, we generalise this lemma from \tame\ \mbox{$T$-graphs} to arbitrary connected graphs:
We show that every connected graph~$G$ admits a \pt\ $(T,\cV)$ for which we can prove an analogue of \cref{lem_highraytop}.
The \pt s that we construct will have two additional properties which are necessary to obtain that analogue of \cref{lem_highraytop}.
The first property is that the \pt\ displays all the ends of the underlying graph.
The second property, called `\ff ', is slightly more technical, but we will see in \cref{lem_WhyFFisUseful} below that this property precisely captures what we need.

Suppose that $G$ is any graph and that $(T,\cV)$ is a \pt\ of~$G$.
As in the previous section, let $\tau\colon\Omega(G)\to \HR(T)\sqcup T$ denote the map that takes each end of~$G$ to the high-ray or point of~$T$ which it corresponds to or lives at, respectively.
An end $\eps$ of~$G$ \emph{lives at} $\uc{t}\subset T$ and \emph{in} $V_{\uc{t}}$ if either the point $\tau(\eps)$ is contained in~$\uc{t}$ or the high-ray $\tau(\eps)$ contains~$t$.
We call $(T,\cV)$ \emph{\ff\ at} an end $\eps$ of~$G$ if $\eps$ corresponds to a high-ray of~$T$ and for every sequence $(\eps_n)_{n\in\N}$ of ends of~$G$ that live in the up-closures $\uc{s_n}$ of successors $s_n$ of tops of~$\tau(\eps)$ we have convergence $\eps_n\to\eps$ in $\Omega(G)$ as soon as 
for every finite vertex set $X\subset V(G)$ there are only finitely many numbers $n\in\N$ with $N_G(V_{\uc{s_n}})=X$.
If $(T,\cV)$ is \ff\ at every end of~$G$ that corresponds to a high-ray of~$T$, then we call $(T,\cV)$ \emph{\ff }.

Indeed, if a connected graph admits a \pt\ which both displays all its ends and is \ff\, then we can prove an analogue of \cref{lem_highraytop}:

\begin{lemma}\label{lem_WhyFFisUseful}
Let $G$ be any connected graph and let $(T,\cV)$ be a \ff\ \pt\ of~$G$ that displays all the ends of~$G$.
Let $\eps$ and $\eps_n$ $(n\in\N)$ be ends of~$G$, and let $\varrho:=\tau(\eps)$ and $\varrho_n:=\tau(\eps_n)$ be the corresponding high-rays in~$T$.
Let $A\subset\N$ consist of all numbers~$n$ for which $\varrho\subsetneq\varrho_n$, and let $B:=\N\setminus A$.
For each $n\in A$ we write $s_n$ for the successor in~$\varrho_n$ of the top of~$\varrho$ in~$\varrho_n$.
\begin{enumerate}
    \item We have convergence $\eps_n\to\eps$ in~$\Omega(G)$ for $n\in A$ and $n\to\infty$ if and only if $A$ is infinite and for every finite vertex set $X\subset V(G)$ there are only finitely many $n\in A$ such that~$X=N_G(V_{\upcl{s_n}})$.
    \item We have convergence $\eps_n\to\eps$ in~$\Omega(G)$ for $n\in B$ and $n\to\infty$ if and only if for every successor node $t\in\varrho$ there are only finitely many $n\in B$ with $\varrho\cap\varrho_n\subset\mathring{\dc{t}}$.
\end{enumerate}
\end{lemma}

\begin{proof}
(i) The forward implication is evident.
The backward implication follows from the assumption that $(T,\cV)$ is \ff .

(ii) The forward implication holds because~$(T,\cV)$ is a \pt\ and as such has finite adhesion.

For the backward implication, we have to prove that $\eps_n\to\eps$ for $n\in B$ and $n\to\infty$.
For this, let $X\subset V(G)$ be any finite vertex set.
Let $t\in\varrho$ be any successor node such that $X\subset V_{\mathring{\dc{t}}}$.
By assumption, we find a large enough number $N\in\N$ such that $t\in\varrho_n$ for all $n\in B$ with~$n\ge N$.
We claim that $\eps_n\in\Omega(X,\eps)$ for all $n\in B$ with~$n\ge N$.
Indeed, consider any such~$n$, and pick any vertex $v\in V_{t}$.
\cref{lem:findRayForHighRay} yields a ray $R_n\in\eps_n$ that is included in $G[\,V\rest\{t'\in\varrho_n\colon t'\ge t\}\,]$ and starts at~$v$.
Similarly, \cref{lem:findRayForHighRay} yields a ray $R\in\eps$ that is included in $G[\,V\rest\{t'\in\varrho\colon t'\ge t\}\,]$ and starts at~$v$.
Then $R_n\cup R$ is a double ray in~$G$ which avoids~$X$ with one tail in~$\eps_n$ and another in~$\eps$, yielding $\eps_n\in\Omega(X,\eps)$.
\end{proof}

In the remainder of this section, we prove that every connected graph admits a \pt\ with these two properties, \cref{thm:PartitionTree}.
For the proof we need the following lemma:

\begin{lemma}\label{lem:Uforff}
Let $G$ be any connected graph and let $C$ be any connected induced subgraph of~$G$ whose neighbourhood is concentrated in an end $\eps$ of~$G$.
Then there exists a non-empty connected vertex set $U\subset V(C)$ that satisfies all of the following conditions:
\begin{enumerate}
    \item $U$ is either finite or concentrated in~$\eps$.
    \item Every component of $C-U$ has finite neighbourhood in~$G$.
    \item Every sequence $(\eps_n)_{n\in\N}$ of ends of~$G$, where each $\eps_n$ lives in a component $D_n$ of $C-U$ so that the map $\N\ni n\mapsto N_G(D_n)$ is finite-to-one, converges to $\eps$.
    \item $N_G(U)\cap N_G(C)$ is infinite.
\end{enumerate}
\end{lemma}

\begin{proof}
Let $G_C$ be the graph that is obtained from~$G[C\cup N(C)]$ by removing all the edges that run between any two vertices in~$N(C)$. By the star-comb lemma~(\ref{StarCombLemma}), $G_C$ contains either a star or a comb $Z$ attached to~$N(C)$. 
Let $G^+_C$ be the supergraph of $G_C$ in which we make $N(C)$ complete. Let $\eps^+$ be the end of $G^+_C$ containing the rays from the clique induced by~$N(C)$, and note that $N(C)\cup Z$ is concentrated in $\eps^+$ in the graph $G^+_C$.
Now apply \cref{thm:envelope} inside $G^+_C$ to find a connected envelope $U^*$ for $N(C) \cup Z$ that is concentrated in~$\eps^+$ and so that $U:= U^* \cap C$ is connected. 
We verify that $U$ satisfies (i)--(iv):

(i) Suppose that $U$ is infinite. We need to show that for any finite set $X$ of vertices of~$G$, almost all vertices of $U$ are contained in $C(X,\eps)$. 
By assumption, almost all vertices of $N(C)$ are contained in $C(X,\eps)$, and by increasing $X$ if necessary, we may assume that $C(X,\eps)$ is the only component of $G-X$ containing vertices of $N(C)$. Now since $U^*$ is concentrated in $\eps^+$, it follows that almost all vertices of $U^*$ can be connected to $N(C)$ in $G^+_C - X$. 
Then the same is true in $G - X$, and so all these vertices belong to $C(X,\eps)$. Thus, we have shown that almost all vertices of $U^*$ (and hence of $U$) belong to $C(X,\eps)$, as desired.

(ii) Every component of $C - U$ is also a component of $G_C-U^\ast$, and so has finite neighbourhood by the definition of an envelope. Note that it does not matter here whether we take the neighbourhood in~$G$ or in~$G_C$, because $N_G(C)$ is included in~$U^\ast$. 

(iii) Given a sequence $(\eps_n)_{n\in\N}$ as in the statement of~(iii), note first that $U$ must be infinite. Hence~$U$ is concentrated in~$\eps$ by~(i), and so
the claim follows by \cref{lem_concentratedconvergence}.

(iv) $N_G(U)\cap N_G(C)$ is infinite by the choice of $Z \subset U^*$.
\end{proof}

\begin{customthm}{\ref{thm:PartitionTree}}
Every connected graph has a \ff\ \pt\ that displays all its ends.
\end{customthm}
\begin{proof}
Let $G$ be any connected graph.
We will use the following notation.
Suppose that $(T_1,\cV_1)$ and $(T_2,\cV_2)$ are two \pt s of $G$ and that $T_1$ has a final level.
Furthermore, suppose that $\cV_1=(\,V_t^1\mid t\in T_1\,)$ and that $\cV_2=(\,V^2_t\mid t\in T_2\,)$.
We write $(T_1,\cV_1)\le (T_2,\cV_2)$~if the following two conditions are met:
\begin{itemize}
    \item $T_2$ extends $T_1$ so that all the points of $T_2\setminus T_1$ lie above the final level of~$T_1$;
    \item $V^1_t=V^2_t$ for all $t\in T_1$ below the final level, and $V_t^1=V^2_{\uc{t}}$ for all $t\in T_1$ in the final level.
\end{itemize}
We will transfinitely construct a sequence $(\,(T_i,\cV_i)\,)_{i\le\kappa}$ of \ff\ \pt s $(T_0,\cV_0)< (T_1,\cV_1) < \cdots$, where $\kappa$ will be an ordinal~$\kappa\le\omega_1$.
Each $T_i$ will have a final level~$F_i$ of limit height at least~$i$, and each $(T_i,\cV_i)$ will display all the ends of~$G$ that do not live at points in~$F_i$.
In the construction, we shall ensure that each point $C\in F_i$ is a connected induced subgraph of~$G$ and $V^i_C$ is equal to the vertex set of this subgraph, i.e., $V^i_C=V(C)$.
We will terminate the construction at the first ordinal $\kappa$ with $T_{\kappa+1}=T_\kappa$.
In the end, we will argue that $(T_\kappa,\cV_\kappa)$ is the desired \pt .

\medskip
To start the construction, we use \cref{lem:maxNT} to let $T_0$ 
be the order tree that arises from an inclusionwise maximal normal tree $T_0'\subset G$ by declaring each component $C$ of $G-T_0'$ a top in~$T_0$ of the down-closed $\omega$-chain~$\dc{N(C)}_{T_0'}$. 
We let $V^0_t:=\{t\}$ for all $t\in T_0'$ and $V^0_C:=V(C)$ for all components $C$ of $G-T_0'$.
Then $(T_0,\cV_0)$ is \ff\ since $T_0$ contains no 
successors of limits.
And it displays all the ends of~$G$ that do not live in any part~$V_C^0$ by \cref{lem:displays}.

At a general step $0<i < \kappa$, suppose we have already constructed $(T_j,\cV_j)$ for all $j<i$ such that $(T_j,\cV_j)<(T_k,\cV_k)$ for all $j<k<i$.
Then we put $T:=\bigcup\,\{\,T_j\mid j<i\,\}$.
We consider two cases, that $i$ is a successor or a limit.

\medskip
\noindent\textbf{Case~1.} In the first case, $i$ is a successor ordinal~$i=j+1$. Then $T=T_j$ has a final level~$F_j$ of limit height.
Consider any point~$C\in F_j$.
The point~$C$ is a top of a high-ray $\varrho$ of~$T$, and for the high-ray $\varrho$ there is a unique end $\eps$ of~$G$ with~$\tau(\eps)=\varrho$ by \cref{lem:displays}.
Furthermore, $C$ is a connected subgraph of~$G$ whose neighbourhood is concentrated in~$\eps$ by \cref{lem:PTtopNbhdClosure}.
We apply \cref{lem:Uforff} to $C\subset G$ and $\eps$ to find a non-empty connected vertex set~$U_C\subset V(C)$ that satisfies the following conditions:
\begin{enumerate}
    \item $U_C$ is either finite or concentrated in~$\eps$.
    \item Every component of $C-U_C$ has finite neighbourhood in~$G$.
     \item Every sequence $(\eps_n)_{n\in\N}$ of ends of~$G$, where each $\eps_n$ lives in a component $D_n$ of $C-U$ so that the map $\N\ni n\mapsto N_G(D_n)$ is finite-to-one, converges to $\eps$. 
    \item $N_G(U_C)\cap N_G(C)$ is infinite.
\end{enumerate}
In each component $D$ of $C-U_C$ we use \cref{lem:maxNT} to pick an inclusionwise maximal normal tree $T(C,D)$ rooted in a vertex that sends an edge to~$U_C$.
Then the neighbourhood in $G$ of every component $K$ of $D-T(C,D)$ is included in $U_C\cup T(C,D)$, where the proportion $N(K)\cap U_C$ is finite (because $N(D)$ is finite by~(ii)) while $N(K)\cap T(C,D)$ is infinite; and in fact the infinitely many neighbours of~$K$ in $T(C,D)$ determine a normal ray $R(K)$~of~$T(C,D)$.

We obtain the tree $T_i$ from $T$ in two steps, as follows.
First, we add for each $C$ in the final level $F_j$ of~$T=T_j$ and every component $D$ of $C-U_C$ the order tree defined by $T(C,D)$ directly above the point~$C$ so that the root of $T(C,D)$ becomes a successor of~$C$.
Second, we add for each $C$ and $D$ every component $K$ of $D-T(C,D)$ as a top of the high-ray~$\dc{R(K)}_{T_i}$.
The family $\cV_i$ is defined as follows.
\begin{itemize}
    \item For each $t\in T-F_j$ we let $V^i_t:=V^j_t$.
    \item For each $C\in F_j$ we let $V^i_C:=U_C$.
    \item For each $t\in T(C,D)$ we let $V^i_t:=\{t\}$.
    \item For each $K$ in the final level of~$T^i$ we let $V^i_K:=V(K)$.
\end{itemize}

Clearly, $(T_j,\cV_j)\le (T_i,\cV_i)$ if $(T_i,\cV_i)$ is a \pt .
We claim that $(T_i,\cV_i)$ is a \pt\ of~$G$.
\ref{pt1} follows with~(iv).
Condition~(ii) ensures that $(T_i,\cV_i)$ has finite adhesion~\ref{pt2}.
\ref{pt3} holds by construction.

To see that the \pt\ $(T_i,\cV_i)$ is \ff , let any end~$\eps$ of~$G$ be given that corresponds to a high-ray $\tau(\eps)$ of~$T_i$.
Here, $\tau$ is defined with regard to~$(T_i,\cV_i)$.
If the order-type of the high-ray $\tau(\eps)$ is equal to the height of~$F_i$, then $(T_i,\cV_i)$ is \ff\ at~$\eps$ for the trivial reason that all tops of~$\tau(\eps)$ lie in the final level~$F_i$ and therefore have no successors.
Hence we may assume that $\tau(\eps)$ is included in~$T_j\subset T_i$.
If the order-type of~$\tau(\eps)$ is less than the height of~$F_j$, then all successors of the tops of~$\tau(\eps)\subset T_i$ are present in~$T_j$.
Hence $(T_i,\cV_i)$ is \ff\ at~$\eps$ because $(T_j,\cV_j)\le (T_i,\cV_i)$ is.
Otherwise, the order-type of~$\tau(\eps)$ is equal to the height of~$F_j$.
Then the successors of the tops of $\tau(\eps)\subset T_i$ are missing in~$T_j$.
To see that $(T_i,\cV_i)$ is \ff\ at~$\eps$, let any sequence $(\eps_n)_{n\in\N}$ of ends of~$G$ be given that live in the up-closures $\uc{s_n}$ of successors of tops of~$\tau(\eps)$ such that the map $\N\ni n\mapsto N_G(V_{\uc{s_n}})$ is finite-to-one.
By construction, each successor $s_n$ is the root of a normal tree~$T(C_n,D_n)$.
Let $X\subset V(G)$ be any finite vertex set.
We have to find a number $N\in\N$ such that all ends $\eps_n$ with $n\ge N$ live in the component $C(X,\eps)$.
Only the vertex sets of finitely many tops $C$ of $\tau(\eps)$ meet $X$, all others must be included in~$C(X,\eps)$.
This partitions the successors $s_n$, and hence the sequence $\eps_n$, into finitely many subsequences: the first subsequence is formed by all ends $\eps_n$ that already live in~$C(X,\eps)$, and the other subsequences are formed by all ends $\eps_n$ that live in $C$ for a common top~$C$ of~$\tau(\eps)$.
Applying~(iii) to all these subsequences except the first yields~$N$.

To show that $(T_i,\cV_i)$ displays all the ends of~$G$ that do not live at points in~$F_i$,
it suffices by \cref{lem:displays} to show that every end $\eps$ of~$G$ with $\tau(\eps)\in T_i$ lives at some~$K\in F_i$ (i.e.\ satisfies $\tau(\eps)=K$).
Here, $\tau$~is defined with regard to~$(T_i,\cV_i)$.
Let any end $\eps$~of $G$ be given with $\tau(\eps)\in T_i$, and recall that $\tau(\eps)$ must be a limit.
Then $\tau(\eps)$ cannot be a point of $T-F_j$, because $(T_j,\cV_j)$ displays all the ends of~$G$ that do not live at points in~$F_j$.
And $\tau(\eps)$ cannot be a point $C\in F_j\subset T_i$, because this would imply $\eps\in\Abs{U_C}$ (by \cref{lem:livingInPoints}) where by~(i) the end $\eps$ would then be sent to the high-ray $\mathring{\dc{C}}_{T_i}$ instead of~$C$.
Therefore, $\tau(\eps)\in F_i$ is the only possibility, as desired.

\medskip
\noindent\textbf{Case~2.} In the second case, $i$ is a limit, so $T$ has no final level (or else the construction would have terminated for some~$j<i$).
Then for every $t\in T$ there is a least ordinal $j(t)<i$ for which $t$ is contained in~$T_{j(t)}$ but not in the final level~$F_{j(t)}$, so $V_t^j=V_t^{j(i)}$ for all $j\in [j(t),i)$.
We let $V_t^i:=V_t^{j(t)}$ for all $t\in T$.
If $C$ is any component of $G-\bigcup_{t\in T}V_t^i$, then
for each $j<i$ there is a unique point $C_j\in F_j$ with $C\subset C_j$, and $B(C):=\{\,C_j\mid j<i\,\}$ is a branch~$T$.
We obtain $T_i$ from $T$ by adding each component $C$ as a top of the branch~$B(C)$ of~$T$.
Letting $V_C:=V(C)$ for all these tops completes the definition of~$(T_i,\cV_i)$.

To ensure that $(T_i,\cV_i)$ is a \pt , we have to show that each $C$ has cofinal down-neighbourhood in~$G/\cV_i$.
Since $(T_i,\cV_i)$ has finite adhesion, it suffices to show that the collection $T_C$ of all points of~$T$ whose parts contain neighbours of~$C$ is an infinite subset of~$B(C)$.
The inclusion $T_C\subset B(C)$ is immediate from the fact that each $G/\cV_j$ with $j<i$ is a $T_j$-graph in which $C_j$ is a maximal vertex.
So assume for a contradiction that $T_C$ is finite and consider any $j<i$ with $T_C\subset T_j-F_j$.
Then $C=C_j$ since otherwise $C\subsetneq C_j$ has a neighbour in~$C_j$ that is contained in no part $V_t$ with $t\in T_C$, contradicting the definition of~$T_C$.
But then $C_j$ has finite neighbourhood, contradicting the fact that $C_j\in F_j$ is a vertex at limit height in the $T_j$-graph~$G/\cV_j$.

To see that $(T_i,\cV_i)$ is \ff , consider any end~$\eps$ of~$G$ for which $\tau(\eps)$ is a high-ray of~$T_i$.
If~$\tau(\eps)$ is a high-ray of any $T_j$ with $j<i$, then we are done because $(T_j,\cV_j)$ is \ff\ at~$\eps$ by assumption.
Otherwise $\tau(\eps)$ is a branch of~$T$ with all tops in~$F_i$, so $(T_i,\cV_i)$ is \ff\  at~$\eps$ for the trivial reason that these tops have no successors in~$T_i$.

To show that $(T_i,\cV_i)$ displays all the ends of~$G$ that do not live at points in~$F_i$, it suffices by \cref{lem:displays} to show that every end $\eps$ of~$G$ with $\tau(\eps)\in T_i$ satisfies~$\tau(\eps)\in F_i$.
Let us assume for a contradiction that $G$ has an end $\eps$ such that $\tau(\eps)=:t\in T_i$ is a point below~$F_i$.
Then $t$ lies below $F_j$ for $j:=j(t)<i$, contradicting our assumption that $(T_j,\cV_j)$ displays all the ends of~$G$ that do not live at points in~$F_j$.

\medskip
We terminate the construction at the first ordinal $\kappa$ with $T_{\kappa+1}=T_{\kappa}$.
Then $\kappa\le\omega_1$ because each $(T_\alpha,\cV_\alpha)$ defines a $T_\alpha$-graph $G/\cV_\alpha$ of finite adhesion that has a final level $F_\alpha$ of height at least~$\alpha$, and these \mbox{$T_\alpha$-graphs} cannot have uncountable branches by \cref{cor_countablebranches}.
By assumption, $(T_\kappa,\cV_\kappa)$ is a \ff\ \pt\ of~$G$ that displays all the ends of~$G$ which do not live at points in~$F_\kappa$.
We claim that $(T_\kappa,\cV_\kappa)$ displays all the ends of~$G$.
For this, it suffices to show that no end~$\eps$ of~$G$ lives at a point in~$F_\kappa$.
And indeed, no end of~$G$ can live at a point in~$F_\kappa$, because otherwise the construction would not have terminated.
Therefore, $(T_\kappa,\cV_\kappa)$ is the desired \pt .
\end{proof}

\section{Proof of \cref{thm_rep}}

\label{sec_7}

\begin{proof}[Proof of \cref{thm:EndspaceTgraph}]
First note that \cref{prop_specialtrees} and \cref{prop_specialendspaces} establish the equivalence $(2) \Leftrightarrow (3)$ in \cref{thm_rep}. Hence, it only remains to show the implication $(1) \Rightarrow (2)$: Every end space is homeomorphic to the end space of a \tame\  graph $G'$ on some order tree $T'$.

Let $\Omega(G)$ be any end space and recall that we may assume $G$ to be connected.
By \cref{thm:PartitionTree} we find a \ff\ \pt\ $(T,\cV)$ of~$G$ that displays all the ends of~$G$.
Without loss of generality all non-limits $t\in T$ are named so that~$V_t=\{t\}$.

\medskip
\noindent\textbf{Construction of~$\boldsymbol{T'}$ and $\boldsymbol{G'}$.} 
Intuitively, we obtain $T'$ from $T$ by splitting up limit nodes of $T$ in a careful manner. Formally, we define an order tree $T'$ and an epimorphism $\varphi\colon T'\to T$ as follows.
Let $L\subset T$ consist of all the limits of~$T$ that have at least one successor in~$T$, and for every $\ell\in L$ let the set $S(\ell)$ consist of all the successors of $\ell$ in~$T$.
For a non-limit $t\in T$ we write $N_t$ for the finite neighbourhood of~$V_{\uc{t}}$ in~$G$.
For each limit $\ell\in L$ we put $\cN_\ell:=\{\,N_t\colon t\in S(\ell)\,\}$.
We obtain the order tree $T'$ from $T$ as follows. First, add for each $\ell\in L$ and $X\in\cN_\ell$ a new node $v(\ell,X)$ that we declare to be a successor of $\ell$ and a predecessor of all $t\in S(\ell)$ with $N_t=X$. Then  delete $L$.

We let the epimorphism $\varphi\colon T'\to T$ be the identity on $T\setminus L$ and we let it send each $v(\ell,X)$ to $\ell$.
Then $\varphi$ is onto. Further, it is a homomorphism (i.e. $t' < s'$ in $T'$ implies $\varphi(t') < \varphi(s')$ in $T$) that is non-injective only at limits, i.e.\ if $t' \neq s'$ but $\varphi(t') = \varphi(s')$, then $t'$ and $s'$ are limits of $T'$ (with $\dc{\mathring{t}'} = \dc{\mathring{s}'}$, by the homomorphism property).
In particular, $\varphi$ defines a bijection $\varrho\mapsto\varphi[\varrho]$ between the high-rays of $T'$ and $T$, which we denote by~$\Phi$.

Finally, let $G'$ be the graph with $V(G'):=T'$ and 
$$E(G') := \set{tt'}:{t < t' \in T' \; \wedge \; \varphi(t)\varphi(t') \in E(\dot{G})}.$$

\medskip
\noindent\textbf{$\boldsymbol{G'}$ is a \tame\ $\boldsymbol{T'}$-graph.}
First, we verify that $G'$ is a $T'$-graph.
From the definition of $E(G')$ it is clear that the end vertices of any edge in $E(G')$ are comparable in $T'$.
Hence, it remains to show that the set of lower neighbours of any point $t\in T'$ is cofinal in~$\dc{\mathring{t}}$. 
Towards this end, let $t' < t$ in $T'$ be arbitrary. We need to find some $x \in T'$ with $t' \leq x < t$ and $xt \in E(G')$. Since $\varphi$ is a homomorphism, $\varphi(t') < \varphi(t)$. Since $\dot{G}$ is a $T$-graph, there is $y \in T$ with $\varphi(t') \leq y < \varphi(t)$ and $y \varphi(t) \in E(\dot{G})$. Since $\varphi$ is onto and level-preserving, there is a unique $x \in T'$ with $t' \leq x < t$ and $\varphi(x) = y$. Then $xt \in E(G')$ as desired.

Then we verify that the $T'$-graph $G'$ is \tame .
For this, let $t = v(\ell,X)$ be any limit of $T'$. Recall that $X$ is a finite set of vertices in~$G$, and we write $\dot{X} \subset T$ for the finitely many nodes in $T$ whose parts intersect~$X$ non-trivially. 
We claim that $S_t := \varphi^{-1}(\dot{X}) \cap \dc{\mathring{t}}$
is as desired. First of all, since $\varphi$ is a homomorphism, $S_t$~is finite. Now we argue that any $t'>t$ in $T'$ has all its down-neighbours below $t$ inside $S_t$. So consider some $x < t < t'$ with $xt' \in E(G')$. Then $\varphi(x) < \varphi(t) < \varphi(t')$ and $\varphi(x) \varphi(t') \in E(\dot{G})$. Now $t' > v(\ell,X)$ in $T'$ implies by construction that $\varphi(x) \in \dot{X}$, and hence $x \in S_t$ as desired.

\medskip
\noindent\textbf{The end spaces are homeomorphic.}
Since $(T,\cV)$ is a sequentially faithful partition tree for $G$, we have a natural bijection $\tau\colon\Omega(G)\to\HR(T)$ from the ends of $G$ to the high-rays of $T$ such that \cref{lem_WhyFFisUseful} provides a combinatorial description of convergence of a sequence of ends $\eps_n \to \eps$ in terms of their associated high-rays $\tau(\eps_n)$ and $\tau(\eps)$.

Similarly, since $G'$ is a uniform $T'$-graph, we have a natural bijection $\sigma\colon\Omega(G')\to\HR(T')$ from the ends of $G'$ to the high-rays of $T'$ such that \cref{lem_highraytop} provides a combinatorial description of convergence of a sequence of ends $\eps'_n \to \eps'$ in terms of their associated high-rays $\sigma(\eps'_n)$ and $\sigma(\eps')$.

To complete the proof, we show that the bijection $\Phi \colon \HR(T') \to \HR(T)$ lifts to a homeomorphism $$f:= \tau^{-1} \circ \Phi \circ \sigma \colon \Omega(G') \to \Omega(G)$$ as in the following diagram:
\begin{center}
\begin{tikzcd}
{}&\arrow[l,phantom,"\Omega(G')\ni\; \; "]\eps'\arrow[r, mapsto,"\sigma"]\arrow[d, mapsto, dashed, "f"']&\sigma(\eps')\arrow[d, mapsto,"\Phi"]\arrow[r,phantom,"\in\HR(T')"]&{}\\
{}&\arrow[l,phantom,"\Omega(G)\ni\;"]\eps& \varphi[\sigma(\eps')]\arrow[l, mapsto,"\tau^{-1}"]\arrow[r,phantom,"\;\in\HR(T)"]&{}
\end{tikzcd}
\end{center}
Towards this aim, consider ends $\eps'_n$ (for $n \in \N$) and $\eps'_\star$ in $\Omega(G')$, with images $\eps_s := f(\eps'_s)$ in $\Omega(G)$ for $s \in \N \cup \{\star\}$. We show that $\eps'_n \to \eps'_\star$ in $\Omega(G')$ if and only if $\eps_n \to \eps_\star$ in $\Omega(G)$.
Write $\varrho'_s:=\sigma(\eps'_s)$ $\HR(T')$ and $\varrho_s:=\tau(\eps_s)$ in $\HR(T)$ for $s \in \N \cup \{\star\}$ for the associated high-rays. By definition of $f$, we have $\varrho_s := \Phi(\varrho'_s)$ for all $s \in \N \cup \{\star\}$. By definition of $\Phi$, we have
$$\set{n \in \N}:{\varrho_\star \subsetneq \varrho_n} = A = \set{n \in \N}:{\varrho'_\star \subsetneq \varrho'_n}.$$
For each $n\in A$ let $s_n$ denote the successor in~$\varrho_n$ of the top of~$\varrho$ in~$\varrho_n$. Then
\begin{align*}
  & \eps'_n \to \eps'_\star \in \Omega(G') \; \text{as } n \to \infty \text{ for } n \in A \\
  \Leftrightarrow \; & |A| = \infty \text{ and for every top } t=v(\ell,X) \text{ of } \varrho'_\star \text{ there are only finitely many } n \in A \text{ with } t \in \varrho'_n \\
  \Leftrightarrow \; & 
   |A| = \infty \text{ and for every finite } X\subset V(G) \text{ there are only finitely many } n\in A \text{ with } X=N_{s_n}\\
  \Leftrightarrow \; & \eps_n \to \eps_\star \in \Omega(G) \; \text{as } n \to \infty \text{ for } n \in A,
\end{align*}
where the first equivalence is \cref{lem_highraytop}~(i) and the third is \cref{lem_WhyFFisUseful}~(i). 
To see the backward implication of the second equivalence, consider any top $t=v(\ell,X)$ of~$\varrho'_\star$. For each $\varrho'_n$ with $t\in\varrho'_n$, we have that ($\ell$ is the predecessor of~$s_n$ and) $N_{s_n}=X$ by definition of~$\Phi$.
Since there are only finitely many $n\in A$ with $X=N_{s_n}$ by assumption, the latter implies that there are only finitely many $n\in A$ with $t\in\varrho'_n$.
To see the forward implication of the second equivalence, consider any finite $X\subset V(G)$.
For each $\varrho_n$ we have that $N_{s_n}$ meets~$V_{\ell_n}$ where $\ell_n$ is the top of~$\varrho$ in~$\varrho_n$, but $N_{s_n}$ avoids $V_\ell$ for all other tops $\ell$ of~$\varrho$, because $\dot{G}$ is a $T$-graph.
Hence $N_{s_n}=X$ for some $n\in A$ implies that only those $\varrho_m$ with $\ell_m=\ell_n$ can possibly satisfy $N_{s_m}=X$.
Therefore, we have $N_{s_m}=X$ if and only if $v(\ell_n,X)\in\varrho'_m$, and by assumption there are only finitely many such $m$ as desired.

Similarly, setting $B= \N \setminus A$, we get 
\begin{align*}
  & \eps'_n \to \eps'_\star \in \Omega(G') \; \text{as } n \to \infty \text{ for } n \in B \\
   \Leftrightarrow \; & 
   |B| = \infty \text{ and for every successor } t\in\varrho'_\star \text{ there are only finitely many } n\in B \text{ with } \varrho'_\star \cap\varrho'_n\subset\mathring{\dc{t}}_{T'} \\
  \Leftrightarrow \; & 
   |B| = \infty \text{ and for every successor } t\in\varrho_\star \text{ there are only finitely many } n\in B \text{ with } \varrho_\star\cap\varrho_n\subset\mathring{\dc{t}}_T \\
   \Leftrightarrow \; & \eps_n \to \eps_\star \in \Omega(G) \; \text{as } n \to \infty \text{ for } n \in B,
\end{align*}
where the first equivalence is \cref{lem_highraytop}~(ii), the second is evident by the properties of~$\Phi$, and the third is \cref{lem_WhyFFisUseful}~(ii). Together, we have $\eps'_n \to \eps'_\star$ in $\Omega(G')$ if and only if $\eps_n \to \eps_\star$ in $\Omega(G)$ as desired.
\end{proof}

\section{Applications}
\label{sec_8}

\subsection[Applications of the structure theorem]{Applications of \cref{thm:PartitionTree}}

Recall that a subgraph $H$ of $G$ is \emph{end-faithful} if mapping every end of $H$ to the end of $G$ including it defines a bijection between the ends of $H$ and~$G$.
Halin conjectured that every connected graph has an end-faithful spanning tree as a subgraph.
Halin confirmed his conjecture for all countable graphs, and Polat confirmed it for all graphs without a subdivision of the $\aleph_1$-regular tree~\cite{polat1997end}, see also~\cite{NormallyTraceable}.
However, Halin's conjecture was refuted in the early 1990's, independently by Seymour and Thomas~\cite{seymour1991end} and by Thomassen \cite{thomassen1992infinite}.

Carmesin~\cite{carmesin2014all} showed that Halin's conjecture becomes true if one asks the spanning tree only to be faithful to ends of a specific type: those that appear as points at infinity in the Freudenthal boundary.
This raises the question whether there is another way to amend Halin's conjecture, one that works for the full array of ends of graphs.

Using \cref{thm:PartitionTree}, we provide a positive answer: 
we show that Halin's end-faithful spanning tree conjecture becomes true if we relax `trees' to `$T$-graphs' and `subgraph' to `contraction-minor', but keep all ends.
Let $H$ be a minor of $G$ with branch sets $V_h$ ($h\in V(H)$).
A ray $h_0 h_1\ldots$ in $H$ \emph{tracks} a ray $v_0 v_1\ldots$ in $G$ if $v_{n_k}\in V_{h_k}$ for some subsequence $(v_{n_k})$ of $(v_n)$.
By standard arguments, for every end $\eps$ of $H$ there exists a unique end $\hat\eps$ of $G$ such that every ray in $\eps$ tracks a ray in~$\hat\eps$.
We say that $H$ is an \emph{end-faithful} minor of~$G$ if the map $\eps\mapsto\hat\eps$ is a bijection $\Omega(H)\to\Omega(G)$.
And we say that $H$ is a \emph{topologically end-faithful} minor of $G$ if the map $\eps\mapsto\hat\eps$ is a homeomorphism $\Omega(H)\to\Omega(G)$.

\begin{corollary}\label{main_Halin}
    Every connected graph contains a $T$-graph of finite adhesion as an end-faithful contraction-minor for some semi-special order-tree~$T$.
\end{corollary}

\begin{proof}
    Let $G$ be a connected graph; we are to show that $G$ contains a $T$-graph of finite adhesion as an end-faithful contraction minor for some semi-special order tree~$T$.

    By \cref{thm:PartitionTree} we find a \ff\ \pt\ $(T,\cV)$ of~$G$ that displays all the ends of~$G$.
    Let $\dot G=G/\cV$ be the contraction minor of $G$ with vertex-set $T$ and branch sets $V_t$ ($t\in T$).
    Then $\dot G$ is a $T$-graph by \ref{pt1}, and it has finite adhesion (as pointed out below \cref{def_partition_tree}).
    Hence $T$ is semi-special by \cref{prop_specialtrees}.

    To see that the contraction minor $\dot G$ of~$G$ is end-faithful, we consider the usual map $\sigma\colon\Omega(\dot G)\to\HR(T)\sqcup T$, and let $\dot\eps$ be any end of~$\dot G$.
    As $\dot G$ is a $T$-graph of finite adhesion, $\sigma$ restricts to a bijection $\Omega(\dot G)\to\HR(T)$ by \cref{SigmaBijection}.
    Let $\varrho:=\sigma(\dot\eps)$ be the high-ray of $T$ that $\dot\eps$ corresponds to.
    We choose a cofinal $\omega$-chain $t_0<t_1<\cdots$ in $\varrho$ arbitrarily.
    By \cref{lem_Tgraphproperties}~\ref{itemT4}, there is a ray $\dot R\subseteq \dot G[\varrho]$ that traverses the nodes $t_i$ in increasing order.
    Then $\dot R\in \dot\eps$ as the end of $\dot R$ corresponds to $\varrho$ (as $\dot\eps$ does) and $\sigma$ is injective.
    By \cref{lem:findRayForHighRay}, there is a ray $R\subseteq G[V_\varrho]$ that is tracked by~$\dot R$ and which corresponds to~$\varrho$.
    Hence $\tau(\eps)=\varrho$ for the end $\eps\in\Omega(G)$ that contains~$R$, where $\tau\colon\Omega(G)\to\HR(T)\sqcup T$.
    In total, $\dot R\in\dot\eps$ tracks $R\in\eps$ and $\eps\in\tau^{-1}(\sigma(\dot\eps))$.
    As $(T,\cV)$ displays all ends of~$G$, the map $\tau$ restricts to a bijection $\Omega(G)\to\HR(T)$, and we saw above that $\sigma\colon\Omega(\dot G)\to\HR(T)$ is bijective, so $\tau^{-1}\circ\sigma\colon\Omega(\dot G)\to\Omega(G)$ is bijective. Therefore, $\dot G$ is an end-faithful minor of $G$.
\end{proof}

\begin{prob}
\label{Halin_open}
   Does every connected graph $G$ contain $T$-graph $H$ as a topologically end-faithful minor for some (semi-)special order-tree~$T$?
\end{prob}

A \emph{separation} of a graph $G$ is an unordered pair $\{A,B\}$ such that $A\cup B=V(G)$ and there is no edge in $G$ between $A\setminus B$ and $B\setminus A$.
We refer to $A\cap B$ as the \emph{separator} of $\{A,B\}$.
The cardinal $|A\cap B|$ is the \emph{order} of~$\{A,B\}$.
If $\{A,B\}$ has finite order, then $\{\Abs{A},\Abs{B}\}$ is a clopen bipartition of $\Omega(G)$.
Two separations $\{A,B\}$ and $\{C,D\}$ of $G$ are \emph{nested} if, possibly after exchanging the name $A$ with $B$ or $C$ with~$D$, we have $A\subset C$ and $B\supseteq D$.
Let $M$ be a set of separations of~$G$.
We say that $M$ is \emph{nested} if its elements are pairwise nested.
We say that $M$ \emph{distinguishes} two ends $\eps_1,\eps_2$ of $G$ if there is a finite-order separation $\{A,B\}\in M$ with $\eps_1\in\Abs{A}$ and $\eps_2\in\Abs{B}$ or vice versa.

\cref{thm:PartitionTree} implies the following deep result by Carmesin~\cite[5.17]{carmesin2014all} which he proved on 30 pages.

\begin{corollary}
  \label{thm_carmesin}
Every graph $G$ has a nested set $M$ of finite-order separations such that $M$ distinguishes every two ends of~$G$.
\end{corollary}

\begin{proof}
Without loss of generality, $G$ is connected.
By \cref{thm:PartitionTree}, $G$ has a \ff\ \pt\ $(T,\cV)$.
For every successor $t\in T$ we let
\[
    S_t:=N_G(\,V_{\uc{t}}\,)\quad\text{and}\quad A_t:=V_{\uc{t}}\cup S_t\quad \text{and}\quad B_t:=V\rest (T\setminus\uc{t}).
\]
Then $\{A_t,B_t\}$ is a finite-order separation of~$G$ with separator equal to~$S_t$ by~\ref{pt2}.

We claim that the set $M$ of all these finite order separations $\{A_t,B_t\}$ is nested.
If $t<t'$ then $A_{t'}\subseteq A_t$, and so $B_{t'}\supseteq B_t$.
If $t$ and $t'$ are incomparable, then $A_t\subseteq B_{t'}$ and $B_t\supseteq A_{t'}$.
Hence $M$ is nested.

Finally, we show that $M$ distinguishes every two ends of~$G$.
For each $\{A_t,B_t\}$ let $\Omega_t:=\Abs{A_t}$.
Since $(T,\cV)$ is \ff , the map $\tau$ is a bijection between $\Omega(G)$ and~$\HR(T)$.
Each set $\Omega_t$ consists precisely of those ends $\eps$ of~$G$ whose corresponding high-ray $\tau(\eps)\subset T$ includes $\dc{t}$ as an initial segment.
Now let $\eps_1,\eps_2$ be two distinct ends of~$G$.
Without loss of generality, $\tau(\eps_1)$ is not included in~$\tau(\eps_2)$.
Let $t'$ be the least element of~$\tau(\eps_1)\setminus\tau(\eps_2)$, and let $t$ be the successor of~$t'$ in the high-ray~$\tau(\eps_1)$.
Then $\eps_1\in\Omega_t$ while $\eps_2\notin\Omega_t$, so $\{A_t,B_t\}$ distinguishes $\eps_1$ and~$\eps_2$.
\end{proof}

\subsection{Applications of \cref{thm_rep} assertion (2)}

\begin{corollary}
\label{cor_main_1}
Forbidding uncountable clique minors does not reduce the complexity of end spaces, i.e., every end space is homeomorphic to the end space of a graph without an uncountable clique minor.
\end{corollary}
\begin{proof}
 This follows from \cref{thm_rep}~(2) by the fact that uniform graphs on special order trees do not contain uncountable clique minors, \cref{cor_unctblclique}. 
\end{proof}
In contrast, forbidding a subdivision of a countably infinite clique gives a normal spanning tree by a result of Halin \cite[Theorem~10.1]{halin1978simplicial}, and so in this  case the end space is metrizable.

The \emph{degree}
of an end is the supremum of the sizes of collections of pairwise disjoint rays in it; Halin~\cite{H65} showed that this supremum is always attained. 

\begin{corollary}
\label{cor_main_2}
Forbidding ends of uncountable degree does not reduce the complexity of end spaces, i.e., every end space is homeomorphic to the end space of a graph in which every end has countable degree.
\end{corollary}
\begin{proof}
  This follows from \cref{thm_rep}~(2) by the fact that all ends of uniform graphs on order trees have just countable degree, \cref{cor_ctbldegree}.
\end{proof}

\subsection{Applications of \cref{thm_rep} assertion (3)}

A \emph{discrete expansion of length $\sigma$} of a topological space $X$ is an increasing sequence $(X_i \colon i < \sigma)$ of non-empty closed subsets of $X$ such that
\begin{itemize}
    \item $X = \bigcup_{i < \sigma} X_i$,
    \item $X_0$ and $X_{i +1} \setminus X_i$ are discrete for all $i+1<\sigma$, and
    \item $X_\ell = \closure{\bigcup_{i < \ell} X_i}$ for all limits $\ell < \sigma$.
\end{itemize}
It is trivial that every Hausdorff space $X$ has a discrete expansion of length $|X|$, see \cite[Remark~7.3]{polat1996ends2}. 
The following remarkable theorem that end spaces have ``short" expansions is a deep result by Polat \cite[Theorem~8.4]{polat1996ends2} which he proved on twenty pages.  

\begin{corollary}
Every end space admits a discrete expansion of length at most $\omega_1$.
\end{corollary}

\begin{proof}
By \cref{thm:EndspaceTgraph}(3) it suffices to show the assertion for ray spaces $\cR(T)$ of special order trees $T$. 
For every limit $\ell < \omega_1$ let $X_\ell \subset \cR(T)$ be the set of high-rays that belong to $T^{<\ell}$. Since each high-ray of $T$ has countable order type, 
$\cR(T) = \bigcup_\ell X_\ell$ is an increasing cover of closed sets.
To get a discrete development, for every limit $\ell < \omega_1$ and every $n \in \N$ we now define a set $X_{\ell + n}$ with $X_{\ell} \subset X_{\ell + n} \subset X_{\ell + \omega}$ as follows:
For every node $t$ of $T^{\ell + n}$ choose, if possible, one high-ray $\varrho_t \in X_{\ell+\omega}$ with $t \in \varrho_t$ (equivalently: $\varrho_t \in [t]$ in the language of \cref{lem_standardbase}).
Let $X_{\ell + n}$ consist of all high-rays in the previous set $X_{\ell + n -1}$ together with all the chosen~$\varrho_t$. 

Then 
$ (X_i \colon i < \omega_1) $
is the desired discrete development. Indeed, all sets are closed by construction, $X_{i+1} \setminus X_i$ is discrete by construction, and $X_\ell = \closure{\bigcup_{i < \ell} X_i}$ for all limits $\ell < \omega_1$.
\end{proof}

\subsection{Acknowledgement.} We thank Ruben Melcher for fruitful discussions on the topic of end spaces which have led to the statement of \cref{lem_concentratedconvergence}.

\bibliographystyle{abbrv}
\bibliography{TreeBib}
\end{document}